\documentclass[]{siamonline190516}

\usepackage[utf8]{inputenc}

\usepackage[left=3cm,right=3cm,bottom=3cm,top=3cm]{geometry}

\usepackage[english]{babel}
\usepackage[T1]{fontenc}
\usepackage{amsmath}
\usepackage{amsfonts}
\usepackage{amssymb}
\usepackage{xargs}
\usepackage{float}
\usepackage{adjustbox}
\usepackage{bm}
\usepackage{bbm}
\usepackage[numbers]{natbib}
\usepackage{dsfont}
\usepackage[toc,page]{appendix}

\usepackage{enumitem}

\usepackage[framemethod=tikz]{mdframed}

\usepackage{lipsum}

\usepackage{color}
\definecolor{myblue}{rgb}{0.122, 0.435, 0.698}
\newmdenv[innerlinewidth=0.5pt, roundcorner=4pt,linecolor=myblue,innerleftmargin=6pt,
innerrightmargin=6pt,innertopmargin=6pt,innerbottommargin=6pt]{bluebox}
\definecolor{myred}{rgb}{0.8, 0.1, 0.1}
\newmdenv[innerlinewidth=0.5pt, roundcorner=4pt,linecolor=myred,innerleftmargin=6pt,
innerrightmargin=6pt,innertopmargin=6pt,innerbottommargin=6pt]{redbox}

\usepackage{stmaryrd}

\SetSymbolFont{stmry}{bold}{U}{stmry}{m}{n}

\usepackage{cleveref}

\usepackage{algorithm}
\usepackage{algpseudocode}

\usepackage{phdalgo}

\newcommand{\mc}{\mathcal}
\newcommand{\eqDef}{:=}

\newcommand{\norm}[1]{\| #1\|}

\newcommand{\hh}{\hspace{-2pt}}

\newcommand{\Ker}{\operatorname{Ker}}

\newcommand{\eqd}{:=}

\renewcommand{\citep}[1]{\citeauthor{#1} (\citeyear{#1})}

\newcommand{\N}{\mathcal{N}} 
\newcommand{\T}{\mathcal{T}} 
\newcommand{\X}{\mathcal{X}} 

\newcommand{\lnu}{\underline{\nu}} 

\newcommand{\rr}{\mathbb{R}} 
\newcommand{\nit}{\mathbb{N}}

\newcommand{\Y}{\mathcal{Y}}
\newcommand{\yy}{\bm{y}}

\newcommand{\x}{x}
\newcommand{\xx}{\bm{\x}}

\newcommand{\lbx}{\underline{x}}
\newcommand{\ubx}{\overline{x}}

\newcommand{\ee}{\bm{e}}

\newcommand{\pp}{\bm{p}}
\renewcommand{\P}{\mathcal{P}}

\newcommand{\mr}{\mathrm}

\newcommand{\ol}{\overline}

\newcommand{\xE}{ {\xx} }
\newcommand{\xP}{ {\yy} }
\newcommand{\xEinf}{ {\xE^\infty} }
\newcommand{\xPinf}{ {\xP^\infty} }

\newcommand{\Tin}{\T^{\circ}}

\newcommand{\omu}{\overline{\mu}}
\newcommand{\umu}{\underline{\mu}}

\newcommand{\nt}{_{n,t}}

\newcommand{\Popt}{\P_{\mr{D}}} 
\newcommand{\Q}{\mathcal{Q}}
\newcommand{\tr}{\top} 

\newcommand{\U}{\mathcal{U}}
\newcommand{\dsoneN}{\mathds{1}_{\hspace{-1pt}N} }
\newcommand{\dsoneT}{\mathds{1}_{\hspace{-1pt}T} }
\newcommand{\EE}{\bm{E}}

\newcommand{\dsone}{\mathds{1}}

\newcommand{\oneT}{
\bm{1}}

\newcommand{ \llbx}{\bm{\lbx}}
\newcommand{ \uubx}{\bm{\ubx}}

\newcommand{\kexp}{^{(k)}}

\newcommand{\epscvg}{\varepsilon_{\rm{cvg}}}
\newcommand{\epsdis}{\varepsilon_{\rm{dis}}}

\newcommand{\llam}{\bm{\lambda}}
\newcommand{\mmu}{\bm{\mu}}
\newcommand{\nnu}{\bm{\nu}}

\renewcommand{\L}{\mathcal{L}}

\newcommand{\uH}{\underline{\T}}
\newcommand{\oH}{\overline{\T}}

\newcommand{\hatt}{\hat{t}}
\newcommand{\nth}{_{n ,\hatt}}

\newcommand{\txt}{\textstyle}

\newcommand{\y}{y}

\newcommand{\xxinf}{\xx^{\infty}}
\newcommand{\yyinf}{\yy^{\infty}}

\newcommand{\ssig}{\bm{\sigma}}

\newcommand{\Ypp}{\Y_{\pp}}

\renewcommand{\SS}{\bm{S}}

\newcommand{\Arhs}{A } 
\newcommand{\Bseuil}{B} 

\newcommand{\dsg }{\textsc{Disag}}

\newcommand{\ids}{^{(s)}}

\newcommand{\bb}{\bm{b}}

\newcommand{\n}{n}

\newcommand{\CutSet}{\Lambda}

\newcommand{\cc}{\bm{c}}

\newcommand{\aff}{ \operatorname{aff} }

\newcommand{\Bp}{B}  

\newcommand{\diag}{\operatorname{diag}}

\newcommand{\oneH}{\oneT_{\bar{\T}}}
\newcommand{\tm}{\text{-}}

\newcommand{\V}{\mathcal{V}}
\newcommand{\E}{\mathcal{E}}
\newcommand{\G}{\mathcal{G}}

\newcommand{\Kexp}{^{(K)}}
\newcommand{\bxx}{ \bar{\xx} }

\newcommand{\ppv}{p^{\textsc{pv}}}
\newcommand{\pppv}{\bm{p}^{\textsc{pv}}}

\newcommand{\pg}{p^g}
\newcommand{\ppg}{\bm{p}^g}
\newcommand{\pgub}{\overline{p}^g}
\newcommand{\pglb}{\underline{p}^g}
\newcommand{\I}{\mathcal{I}}
\newcommand{\xst}{\bo^{\textsc{st}}}
\newcommand{\Cst}{C^{\textsc{st}}}
\newcommand{\xon}{\bo^{\textsc{on}}}
\newcommand{\normop}[1]{\left\Vert| #1|\right\Vert}\newcommand{\bo}{b} %

\newcommand{\scal}{\kappa}
\newcommand{\Kexpp}{^{(K-1)}}

\newcommand{\Tcut}{\T}
\newcommand{\Ncut}{\N}
\newcommand{\Nall}{[N]}
\newcommand{\Tall}{[T]}
\newcommand{\Tun}{\underline{\T}_n}
\newcommand{\Ton}{\ol{\T}_n}
\newcommand{\barT}{\bar{\T}}

\newcommand{\scvg}{s^{\text{cv}}}

\usepackage[]{subfig}
\usetikzlibrary{intersections,calc,arrows,positioning}

\usepackage[disable,colorinlistoftodos,bordercolor=orange,backgroundcolor=orange!20,linecolor=orange,textsize=scriptsize]{todonotes}

\RequirePackage{amssymb}

\renewcommand{\geq}{\geqslant}
\renewcommand{\leq}{\leqslant}

\Crefname{algocf}{Algorithm}{Algorithms}

\Crefname{claim}{Claim}{Claims}

\newtheorem{assumption}{Assumption}

\theoremstyle{definition}

\newsiamremark{remark}{Remark}

\graphicspath{{../../images/},{../}}  
\DeclareGraphicsExtensions{.pdf,.jpeg,.png,.svg}

\author{Olivier Beaude \thanks{EDF R\vspace{-0.5pt}\&\vspace{-1pt}D, OSIRIS, Palaiseau, France
  (\email{[olivier.beaude,nadia.oudjane]@edf.fr})}
\and  Pascal Benchimol\thanks{CMAP, \'Ecole polytechnique, CNRS and  Triscale Innov, Palaiseau, France (\email{pascal.benchimol@polytechnique.edu}).}
\and St\'ephane Gaubert 
\thanks{Inria and CMAP, \'Ecole polytechnique, CNRS, Palaiseau, France
 (\email{stephane.gaubert@inria.fr}) } 
\and  Paulin Jacquot\thanks{EDF R\vspace{-0.5pt}\&\vspace{-1pt}D, OSIRIS, Inria and CMAP, \'Ecole polytechnique, CNRS, Palaiseau, France, \email{paulin.jacquot@polytechnique.edu}}
\and Nadia Oudjane\footnotemark[2]}

\title{A privacy-preserving method to optimize distributed resource allocation \thanks{Initial version submitted on August~ 2, 2019. Revision of \today. Some of the present results have been announced in the conference proceedings~\cite{paulin2019disaggregationenergy}.}}

\begin{document}

\maketitle


\begin{abstract}
  We consider a resource allocation problem involving a large number of agents with individual constraints subject to privacy, 
and a central operator whose objective is to  
 optimize a global, possibly nonconvex,  cost 
 while satisfying the  agents' constraints,  for instance an energy operator in charge of the  management of energy consumption flexibilities of many individual consumers. 
We provide a privacy-preserving algorithm that does compute the optimal allocation of resources, avoiding each agent to reveal her private information
(constraints and individual solution profile) neither to the central operator nor to a third party. 
Our method
relies on an aggregation procedure: we 
compute iteratively a global allocation
of resources, and gradually ensure existence of  a disaggregation, that is individual profiles satisfying agents' private constraints,
by a protocol involving the generation of polyhedral cuts and secure multiparty computations (SMC). To obtain these cuts, we use an alternate projection
method,
which is implemented locally by each agent,
preserving her privacy needs.
We address especially the case in which
the local and global constraints define
a transportation polytope.
Then, we provide theoretical convergence estimates
together with numerical results, showing
that the algorithm
can be effectively
used to solve  the allocation problem in high dimension, while addressing privacy issues. 
\end{abstract}

\section{Introduction}

\subsection{Motivation}
Consider an operator of an electricity microgrid optimizing the joint production schedules of renewable and thermal %
power plants in order to satisfy, at each time period, the consumption constraints of its %
consumers. %
To optimize power generation or market costs and   the integration of renewable energies, this operator relies on demand response techniques,
taking advantage of the flexibilities of some of the consumers electric appliances---those which can be controlled without impacting the consumer's comfort, as electric vehicles or water heaters~\cite{PaulinTSG17}. %
 However, for confidentiality reasons, consumers may not be willing to provide their consumption constraints or their consumption profiles to a central operator or to any third party, as this information could be used to infer private information such as their presence at home.

 The \emph{global problem} of the operator is to find an allocation  of power (aggregate flexible consumption) $\pp=(p_t)_{t\in\Tall}$, where $\Tall:=\{1,\dots,T\}$
 is the set of time periods,
 such that $\pp \in \P$ (feasibility constraints of power allocation, induced by the power plants constraints), while minimizing the cost  $f(\pp)$ (representing the production and operations costs of the centrally controlled power plants).
 Besides, this aggregate allocation has to match an individual consumption profile $\xx_n= (x\nt)_{t\in\Tall} $ for each of the consumer (agent) $n\in\Nall$ considered.
 The problem can be written as follows:
\begin{subequations}
\label{pb:global}
\begin{align} 
& \min_{\xx\in \rr^{N \times T}\hh,\  \pp\in \P} f(\pp)  \\
 & \xx_n \in \X_n, \   \forall n \in \Nall  \label{cons:indfeas}\\
 &  \sum_{n\in\Nall} x\nt = p_t, \ \forall t \in \Tall \label{cons:disagfeas} \  ,
\end{align}
\end{subequations}
The (aggregate) allocation $\pp$ can be made \emph{public}, that is, revealed to all agents. However, the individual constraint set $\X_n$ and individual profiles $\xx_n$ constitute \emph{private} information of agent $n$, and should not be revealed to the operator or any third party. %

It will be helpful to think of Problem \eqref{pb:global} as the combination
of two interdependent subproblems:
\begin{enumerate}[wide,label=\roman*)]
\item given
an aggregate allocation $\pp$, the \emph{disaggregation problem}
consists in finding, for each agent $n$,
an
individual profile $\xx_n$ satisfying her individual constraint \eqref{cons:indfeas}, so that constraint~\eqref{cons:disagfeas} is satisfied. This is equivalent
to:
\begin{subequations}\label{pb:disag}
 \begin{align}  \label{pb:disag-find}
& \textsc{Find } \xx \in \Y_{\pp} \cap \X  \\
 \text{ where } & \Y_{\pp} \eqd \{\yy\in\rr^{NT} | \yy^\tr \dsoneN  = \pp \} \label{eq:defY}  \hfill \text{ and } \hfill  \X \eqd \prod_{n\in\Nall} \X_n \ . 
\end{align}
\end{subequations}
When  \eqref{pb:disag} has a solution, we say that a \emph{disaggregation} exists for $\pp$;  %
\item For a given subset $\Q\subset \P$, we define the \emph{master problem}, 
\begin{equation} \label{pb:master-Q}
\min_{\pp\in \Q} f(\pp) \ .
\end{equation}
\end{enumerate}
When $\Q$ is precisely the set of aggregate allocations for which a disaggregation exists, the optimal
solutions of the master problem correspond to the optimal solutions of \eqref{pb:global}.

Aside from the example above, \emph{resource allocation problems} (optimizing common resources shared by multiple agents) with the same structure as \eqref{pb:global}, find many applications in energy \cite{muller2017aggregation,PaulinTSG17}, %
 logistics \cite{laiLam95shipping}, distributed computing \cite{ma1982task}, health care \cite{rais2011operations} and telecommunications  \cite{zulhasnine2010efficient}. 
 In these applications, several entities or agents (e.g.\ consumers, stores, tasks) share a common resource (energy, products, CPU time, broadband) which has a global cost for the system. %
 For large systems composed of multiple agents, the dimension of the overall problem can be prohibitive. Hence, a solution is to rely on decomposition and distributed approaches~\cite{ bertsekas1989parallel,palomar2006tutorial,xiao2006optimal}.
Besides, %
agents' individual constraints are often subject to privacy issues  \cite{huberman2005valuating}. %
These considerations have paved the way to the development of privacy-preserving, or non-intrusive methods and algorithms, e.g.~\cite{zoha2012non,jagannathan2006new}. 

In this work, except in \Cref{sec:generArbitraryIndividualSet}, we consider that each agent $n\in\Nall$ has a global demand constraint (e.g.\ energy demand or product quantity), which confers to the disaggregation problem the particular structure of a transportation polytope~\cite{bolker1972transportation}: the sum over the agents is fixed by the aggregate solution $\pp$, while the sum over the $T$ resources is fixed by the agent global demand constraint. Besides, individual constraints can also include minimal and maximal levels on each resource. For instance, electricity consumers require, through their appliances, a minimal and maximal power at each time period.

\subsection{Main Results} The main contribution of the paper is to provide a non-intrusive and distributed algorithm (\Cref{algo:confidOptimDisag}) that computes an aggregate resource allocation $\pp$,  optimal solution of the---possibly nonconvex---optimization problem \eqref{pb:global}, along with feasible individual profiles $\xx$ for agents, without revealing the individual constraints of each agent to a third party, either another agent or a central operator. 
The algorithm solves iteratively instances of \emph{master  problems} $\min_{\pp \in\P^{(s)}} f(\pp) $, %
obtained by constructing a decreasing sequence of successive approximations $\P^{(s)}$
of the set of aggregate consumptions $\pp$
for which a disaggregation exists (see~\eqref{pb:global}),
starting from $\P^{(0)}=\P$.
At each step, we reduce the set  $\P^{(s)}$ by incorporating a new constraint on $\pp$ (a {\em cutting plane}), before solving the next master problem. 
We shall see that this cutting plane can be computed and added to the master problem without revealing any individual information on the agents.

More precisely, to identify whether or not disaggregation \eqref{pb:disag} is feasible and to add a new constraint in the latter case, our algorithm relies on the alternate projections method (APM) \cite{von1950functional,GUBIN19671} for finding a point in the intersection of convex sets.
 Here, we consider the two following sets: on the one hand, the affine space of profiles $\xx \in \rr^{NT}$ aggregating to a given resource allocation $\pp$,
 and on the other hand, the set defined by all agents individual constraints (demands and bounds).
 As the latter is defined as a Cartesian product of each agent's feasibility set, APM can operate in a distributed fashion.
  The sequence constructed by APM converges to a single point if the intersection of the convex sets is nonempty, and it converges to a periodic orbit of length $2$ otherwise.
If APM converges to a periodic orbit, meaning that the disaggregation is not feasible,  we construct from this orbit a
polyhedral {\em cut},
i.e.\ a linear inequality satisfied by all feasible solutions $\pp$  of the global 
problem \eqref{pb:global}, but violated from the current resource allocation (\Cref{th:cutN0T0violated}).
 Adding this cut to the \emph{master problem} \eqref{pb:master-Q} by updating $\Q$ to a specific subset, we can recompute a new resource allocation and repeat this procedure until disaggregation is possible.
  At this stage, the use of a cryptographic protocol, {\em secure multiparty computation}, allows us to preserve the privacy of agents.
  Another main result stated in this paper is the explicit  upper bound on the convergence speed of APM in our framework (\Cref{thm:cvgAP}), which is
  obtained by spectral graph theory methods, exploiting also geometric properties of transportation polytopes.  %
  This explicit speed shows a linear impact of the number of agents, which is a strong argument for the applicability of the method in large distributed systems.

\subsection{Related Work} A standard approach (e.g.~\cite{palomar2006tutorial, xiao2004simultaneous,seong2006optimal}) to solve resource allocation problems in a distributed way is to rely on a Lagrangian based decomposition technique: for instance dual subgradient methods \cite[Ch.~6]{bertsekas1997nonlinear} or ADMM \cite{glowinski1975approximation}.
  Such techniques are generally used to decompose a large problems into several subproblems of small dimension. 
However, those methods often %
 require global convexity hypothesis, which are not satisfied in many practical problems (e.g.\ MILP). We refer the reader to \cite[Chapter 6]{bertsekas1997nonlinear}  for more background. 
On the contrary, our method can be used when the  allocation problem \eqref{pb:global} is not convex.

As  explained in \Cref{sec:generArbitraryIndividualSet}, the method proposed here can be related to Bender's decomposition \cite{benders1962partitioning}. 
It differs from Bender's approach in the way cuts are generated: 
instead of solving linear programs, we use APM and our theoretical results, which provides a decentralized, privacy-preserving and scalable procedure.
In contrast, %
 at each stage, Benders' algorithm requires to solve  a  linear program  requiring the knowledge of the private constraints of each individual agent (see \Cref{subsec:benders} for more details).

The problem of the aggregation of constraints has been studied in the field of energy, in the framework of smart grids  \cite{muller2017aggregation,anjos2018decentralized}.  In \cite{muller2017aggregation}, the authors study the management of energy flexibilities and propose to approximate individual constraints by zonotopic sets to obtain an aggregate feasible set. A centralized aggregated problem is solved via a subgradient method, and a disaggregation procedure of a solution computes individual profiles. 
In \cite{anjos2018decentralized}, the authors     propose to solve the economic power dispatch of a microgrid, subject to several agents private constraints, by using a Dantzig-Wolfe decomposition method. %

APM has been the subject of several works in itself \cite{GUBIN19671,bauschke1993convergence,bauschke2015bregman}. The authors of \cite{borwein2014analysis} provide general 
results on the convergence rate of APM for semi-algebraic sets. They show that the convergence is geometric for polyhedra. However, it is generally hard to compute explicitly the geometric convergence rate of APM, as this requires to bound the singular values of certain matrices arising from the polyhedral constraints.
A remarkable example where  an explicit convergence rate for APM  has been
obtained is  \cite{nishihara2014convergence}. The authors consider there a different class of polyhedra arising in submodular optimization. A common point with our results is the use of spectral graph theory arguments to estimate singular values. %

\subsection{Structure}
In~\Cref{sec:resourceallocation-disag}, we describe the class of resource allocation problems that we  address. We formulate the idea of  decomposition via \emph{disaggregation} subproblems. 
In~\Cref{sec:APM}, we focus on APM, the subroutine used to solve the disaggregation subproblems.
 In~\Cref{subsec:generateHoffmanAPM},
 after recalling basic properties of APM,  
 we establish the key result on which relies the proposed decomposition: how to generate a new cut to add in the master problem, from the output of APM. 
 In~\Cref{subsec:APMwithSMC}, we show how to guarantee
 the privacy of the proposed procedure by using secure multiparty computation techniques. 
  In~\Cref{subsec:complexityAnalysis} , we establish an explicit upper bound on the rate of convergence of APM in our case. 
  In \Cref{sec:generArbitraryIndividualSet}, we  generalize part of our results and propose a modified algorithm which applies to polyhedral agents constraints.
 Finally, in~\Cref{sec:examples}, we present numerical examples:  \Cref{subsec:illusToyExample} gives an illustrative toy example in dimension $T=4$, while in  \Cref{subsec:appli-microgrid}, we consider a larger scale, nonconvex example, coming from the microgrid application presented at the beginning of the introduction.

 \textbf{Notation.} Bold fonts, like ``$\xx$'',
 are used to denote vectors, while normal fonts, like ``$x$'', refer to a scalar quantities.
 We denote by $\bm{v}^\tr$ the transpose of a vector $\bm{v}$.
 Recall that we denote by $[N]$ the set $\{1,\dots,N\}$.
Calligraphic letters such as $\Tcut,\Ncut,\X$ are used to denote sets,
and if $\Tcut \subset \Tall$, $\Tcut^c:=\{t\in \Tall \setminus \Tcut\}$ denotes the complementary set of $\Tcut$. 
The notation
$\U([a,b])$ stands for the uniform distribution on $[a,b]$.
The notation $P_\mc{C}(.)$ refers to the Euclidean projection onto a convex set $\mc{C}$.  For $d \in \nit, \mathds{1}_d$ denotes the vector of ones $(1 \dots 1)^\tr \in \rr^d$.

\section{Resource Allocation and Transportation Structure} 
\label{sec:resourceallocation-disag}

\subsection{A Decomposition Method Based on Disaggregation}
As stated in the introduction, we consider a centralized entity (e.g.~an energy operator)  interested in minimizing a possibly nonconvex cost function $\pp \mapsto f(\pp)$, where $\pp \in \rr^T$ is the {\em aggregate allocation} of $T$ dimensional resources (for example power production over $T$ time periods).
This resource allocation $\pp$ is to be shared between
$N$ individual agents, each agent obtaining a part $\xx_n \in \X_n$, where $\X_n$ denotes the individual feasibility set of agent $n$.

The global problem the operator wants to solve is described in  \eqref{pb:global}.
We assume that
in Problem \eqref{pb:global}, the constraints set $\X_n$ and individual profile $\xx_n$ are {\em confidential} to agent $n$ and should not be disclosed to the central operator or to another agent.

Let us define the set $\Popt $ of feasible aggregate allocations that are disaggregable as:
\begin{equation} \label{eq:def-Popt}
\Popt \eqd \big\{ \pp \in \P \ | \ \exists \xx \in \X  \ ; \  \pp= \txt\sum_n \xx_n \big\} \ .
\end{equation} 
Feasibility of problem \eqref{pb:global} is equivalent to having $\Popt$ not empty.

Constraints for each agent are composed of a global demand over the resources and lower and upper bounds over each resource, as given below:

\begin{assumption} \label{assp:Xn}
For each $n\in\Nall$,  there exists fixed parameters $E_n >0$, $\uubx_n\in\rr^T$, $\llbx_n\in\rr^T$ such that :
\begin{equation} \label{eq:XnTransport}
\X_n = \{ \xx_n \in \rr^T: \txt\sum_{t \in\Tall } x\nt=E_{n} \text{ and } \lbx\nt \leq x\nt \leq \ubx\nt\} \ \neq \emptyset.
\end{equation}
\end{assumption}
In particular, $\X_n$ is convex and compact.  
Given an allocation $\pp$, the structure obtained on the matrix $(\xx\nt)\nt$, where sums of coefficients along columns and along rows are fixed, is often referred to as a {\em transportation problem}. The latter has many applications (see e.g.~\cite{aneja1979bicriteria,munkres1957algorithms}).
We focus on transportation problems in \Cref{sec:resourceallocation-disag,sec:APM}, while in \Cref{sec:generArbitraryIndividualSet}, we give a generalization of some of our results in the general case where $\X_n$ is a polyhedron.

Given a particular allocation $\pp \in \P$, the operator will be interested to know if this allocation is \emph{disaggregable}, that is, if there exists individual profiles $(\xx_n)_{n\in\Nall} \in \prod_n \X_n$ summing to $\pp$, or equivalently if the \emph{disaggregation problem} \eqref{pb:disag} has a solution.

 Following  \eqref{pb:disag}, the \textit{disaggregate} profile refers to $\xx$, while the \textit{aggregate} profile refers to the allocation $\pp$.
 Problem \eqref{pb:disag} may not always be feasible. 
Some necessary conditions for a disaggregation to exist, obtained by summing the individual constraints on $\Nall$, are the following \textit{aggregate } constraints:
\begin{subequations}
\label{eq:agg-cond}
\begin{align}  
& \pp^\tr \dsoneT= \EE^\tr \dsoneN \label{eq:agg-cond-sum} \\  
\text{ and } \ & 
 \llbx^\tr \dsoneN \leq \pp \leq \uubx^\tr  \dsoneN \ \label{eq:agg-cond-bounds}.
\end{align}
\end{subequations}
Because they are necessary, we assume that those aggregate conditions \eqref{eq:agg-cond} hold for vectors of $\P$, that is:
\begin{assumption} \label{assp:PverifAggregCond}
All vectors $\pp \in \P$ satisfy \eqref{eq:agg-cond}, that is, $\P \subset \{ \pp \in \rr^T \ | \ \eqref{eq:agg-cond} \text{ hold } \}$.
\end{assumption}
However, conditions \eqref{eq:agg-cond} are not sufficient in general, as explained in the the following section.

\subsection{Equivalent Flow Problem and Hoffman Conditions}

Owing to its special structure, the problem under study can be rewritten as as a flow problem in a graph, as stated in~\Cref{prop:flowProblem}. We refer the reader to the book \cite[Chapter 3]{cook2009combinatorial} for terminology and background. 
\begin{definition} \label{def:flowProblem}
Consider a directed graph $\G=(\V,\E)$ with vertex set $\V$, edge set $\E \subset \V \times \V$, demands  $d:\V \rightarrow \rr$ (where $d_v< 0$ means that $v$ is a production node), edge lower capacities  $\ell:\E \rightarrow \rr_+$ and upper capacities $u: \E \rightarrow \rr_+$. A flow on $\G$ is a function $\xx:\E \rightarrow \rr_+$ such that $\xx$ satisfies the capacity constraints, that is $ \forall e \in \E, \ell_e \leq x_e \leq u_e $,  and Kirchoff's law, that is, $\forall v \in \V, \sum_{e \in\delta_v^+ } x_e =d_v +\sum_{e \in\delta_v^-} x_e $, where $\delta^+_v$ (resp. $\delta^-_v$) is the set of edges ending at  (resp. departing from) vertex $v$.
\end{definition}
The following proposition is immediate:
\begin{proposition} \label{prop:flowProblem}
  Consider the bipartite graph  $\G$ whose set of vertices
  is the disjoint union $\V=\Tall \sqcup \Nall$
  and whose set of edges is $\E=\{(t,n) \}_{t\in \Tall, n \in \Nall}$. Define demands on nodes $\Tall$ by $d_t=-p_t $ and  demands on nodes $\Nall$ by $d_n=E_n$. Assign to each edge $(t,n)$ an upper capacity $ u\nt=\ubx\nt$ and lower capacity $\ell\nt=\lbx\nt$. 
Then, finding a solution $\xx$ to \eqref{pb:disag} is equivalent to finding a feasible flow in $\G$.
\end{proposition}

Hoffman \cite{hoffman1960some} gave a necessary and sufficient   condition for the flow problem to be feasible in a graph with \emph{balanced} demands, that is $d(\V) \eqd \sum_{{v\in\V}} d_v =0$ (total production matches total positive demand, in our case \eqref{eq:agg-cond-sum}). 
This generalizes a result of Gale (1957). 
The stated condition is intuitive: there cannot be a subset of nodes whose
demand exceeds its ``import capacity''. %
\begin{theorem}[{\cite{hoffman1960some}}] \label{th:hoffman}
Given a digraph $\mathcal{G}=(\V,\E)$ with demand $ d \in \rr^\V$ such that $d(\V) = 0$ and capacities $\ell \in (\rr \cup \{-\infty\})^\E$  and $u \in(\rr \cup {\infty})^\E$ with $\ell \leq u$, there exists a feasible flow $\xx \in \E \rightarrow \rr_+ $ on $\G$
if and only if: 
\begin{equation}
\label{eq:gale-condition}
\forall \mc{A} \subset \V, \ \ \sum_{e \in \delta^+({\mc{A}}) } u_e \geq \sum_{v \in \mc{A}} d_v+ \sum_{e \in \delta^+({\mc{A}^c})}  \ell_e \ ,
\end{equation}
where $\delta_+(\mc{A})\eqDef \{(u,v)\in\E | u\in {\mc{A}^c}, v\in \mc{A} \}$ is the set of edges coming to set $\mc{A}$ and ${\mc{A}^c}\eqDef \V \setminus \mc{A}$. 
\end{theorem} 
\Cref{prop:hoffmanDisag} translates \Cref{th:hoffman} in our framework:
\begin{proposition} \label{prop:hoffmanDisag}
  Given an aggregate resource allocation $\pp=(p_t)_{t\in \Tall} \in \P$,
  the disaggregation problem is feasible, meaning that $\pp\in \Popt$,
  iff: 
\begin{equation} \label{eq:hoffmanCondReverse}
\forall \Tcut \subset \Tall, \forall \Ncut  \subset \Nall, \ \  
\sum_{t \in \Tcut} p_t - \sum_{n \in \Ncut } E_n  + \sum_{t \notin \Tcut, n \in \Ncut } \lbx\nt \leq \sum_{t \in \Tcut, n \notin \Ncut }  \ubx\nt \, .
\end{equation}
\end{proposition}
\begin{proof}
We apply \eqref{eq:gale-condition} with $\mc{A} \eqDef \Tcut^c \cup \Ncut^c$ and use the equality $d(\V)=0=\sum_{v\in \mc{A}} d_v+ \sum_{v\in\mc{A}^c} d_v$.
\end{proof}

From \Cref{th:hoffman} or \Cref{prop:hoffmanDisag} above, one can see that the aggregate constraints~\eqref{eq:agg-cond} are in general not sufficient to ensure that the disaggregation problem has a solution. 

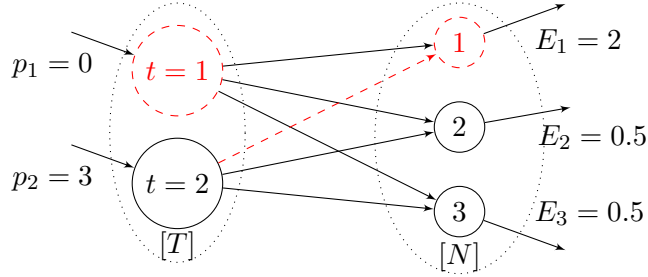
\begin{figure}
\begin{center}
\begin{tikzpicture}[scale=0.75]
\node [draw,circle,red,dashed] (h1) at (0,2) {$t=1$};
\node[above left = 1cm of h1] (p1) {};
\draw [-latex'] ($ (h1) + (160:2) $) -- node [below left] {$p_1=0$} (h1);
\node [draw,circle] (h2) at (0,0) {$t=2$};
\node[above left = 1cm of h2] (p2) {};
\draw [-latex'] ($ (h2) + (160:2) $) -- node [below left] {$p_2=3$} (h2);

\node [draw,circle,red,dashed] (i1) at (5,2.5) {$1$};
\draw [-latex'] (i1) -- node [below right] {$E_1=2$} ($ (i1) + (20:2) $) ;
\node [draw,circle] (i2) at (5,1) {$2$};
\draw [-latex'] (i2) -- node [below right] {$E_2=0.5$} ($ (i2) + (10:2) $) ;
\node [draw,circle] (i3) at (5,-0.5) {$3$};
\draw [-latex'] (i3) -- node [above right] {$E_3=0.5$} ($ (i3) + (-20:2) $) ;

\draw [-latex'] (h1) --  node [text width=2.5cm,midway,above,align=left ] {} (i1);
\draw [-latex'] (h1) --  node [text width=2.5cm,midway,above,align=left ] {} (i2);
\draw [-latex'] (h1) --  node [text width=2.5cm,midway,above,align=left ] {} (i3);
\draw [-latex',red,dashed] (h2) --  node [text width=2.5cm,midway,below,align=left ] {} (i1);
\draw [-latex'] (h2) --  node [text width=2.5cm,midway,below,align=left ] {} (i2);
\draw [-latex'] (h2) --  node [text width=2.5cm,midway,below,align=left ] {} (i3);

\draw [dotted] (0,0.9) ellipse (1.2cm and 2.3cm)  ;
\node at (0,-1.1) {$\Tall$};

\draw [dotted] (5,0.8) ellipse (1.5cm and 2.4cm)  ;
\node at (5,-1.3) {$\Nall$};

\end{tikzpicture}
\end{center}
\caption{Example of a flow representation of the disaggregation problem ($T=2$ and $N=3$, $\llbx=0$, $\uubx=1$). \textit{Here, the aggregate constraints \eqref{eq:agg-cond} are verified, but  condition \eqref{eq:gale-condition} written with $\mc{A}=\{ t_1, \n_1\}$ (dashed nodes) does not hold.} \vspace{-0.5cm}}
\label{fig:graphDisagregationEx}
\end{figure}

For a given set $\Tcut$, there is a choice of $\Ncut$ which leads to the strongest inequality \eqref{eq:hoffmanCondReverse}, namely:
\begin{align} \label{eq:ineq2powTv1}
&\sum_{t \in \Tcut} p_t  \leq  \min_{\Ncut  \subset \Nall} \left\{ \sum_{n \in \Ncut } E_n - \sum_{t \notin \Tcut, n \in \Ncut } \lbx\nt + \sum_{t \in \Tcut, n \notin \Ncut }  \ubx\nt \right\} ,  %
\end{align}
In this way, we get $2^T-2$ inequalities corresponding to the proper subsets $\Tcut \subset \Tall$.
Moreover, in general, these $2^T-2$ inequalities are not redundant. Although this is not stated in \cite{hoffman1960some}, this is a classical result whose proof is elementary.

\section{Disaggregation Based on APM}
\label{sec:APM}

\subsection{Generation of Hoffman's Cuts by APM}
\label{subsec:generateHoffmanAPM}
In this section, we  propose an algorithm that solves \eqref{pb:global}
while preserving the privacy of the agent's constraints $\X_n$ and individual profile $\xx_n \in \rr^T$.
 To do this, the proposed algorithm is implemented in a decentralized manner and relies on the method of {\em alternate projections} (APM) to solve the disaggregation problem \eqref{pb:disag}.

Let us consider the polyhedron enforcing the agents constraints:
\begin{align*}
& \X \eqDef %
\X_1 \times \dots \times \X_N \ %
\text{ where } \X_n \eqDef \left\{ \xx_n \in \rr^{T}_+ \ | \  \textstyle\sum_{t\in \Tall} \x\nt = E_n  \text{ and } \forall t, \  \lbx\nt \leq \x\nt \leq \ubx\nt  \right\} . 
\end{align*}
Besides, given an allocation $\pp \in \P$, we consider the set of  profiles aggregating to $\pp$ :
\begin{align*}
\Y_{\pp} \eqd \big\{ \xx \in \rr^{NT} \ | \ \forall t \in \Tall ,  \txt\sum_{n \in \Nall}  \x\nt = p_t \big\} \ .
\end{align*}
Note that $\Y_{\pp}$ is an affine subspace of $\rr^{NT}$ (to be distinguished from $ \P$ which is a subset of $\rr^T$), and that $\Y_{\pp} \cap \X$  is empty iff $\pp \notin \Popt $, according to the definition of $\Popt$ in \eqref{eq:def-Popt}.
The idea of the proposed algorithm is to build a finite sequence of decreasing subsets $(\P^{(s)})_{ 0 \leq s \leq S}$  such that:
\begin{equation*}
\P =\P^{(0)} \supset \P^{(1)} \supset \dots \supset \P^{(S)} \supset \Popt \ .
\end{equation*}
At each iteration, a new aggregate resource allocation $\pp^{(s)}$ is obtained by solving an instance of the master problem introduced in \eqref{pb:master-Q} with $\Q=\P^{(s)}$:
\begin{subequations} \label{pb:master}
\begin{align}
& \min_{\pp \in \rr^T} f(\pp) \\
\text{s.t.} \ & \pp \in \P^{(s)} \ .
\end{align}
\end{subequations}
In the sequel,
we will refer to~\eqref{pb:master} as the \emph{master problem}. 
Our procedure relies on the following immediate observation:
\begin{proposition}
  If $\pp^{(s)}$ is a solution of \eqref{pb:master}
  and $\xx \in \Y_{\pp^{(s)}}\cap \X $, then $(\pp^{(s)},\xx)$ is an optimal solution of the initial problem \eqref{pb:global}. \hfill \qed
\end{proposition}

Having in hands a solution $\pp^{(s)}$, we can
check whether $\Y_{\pp^{(s)}}\cap \X \neq \emptyset$, and then find
a vector $\xx$ in this set, 
using APM on $\X$ and $\Y_{\pp^{(s)}}$, as described in \Cref{algo:vonNeumannProj} below (where $\Y= \Y_{\pp}$).

\begin{algorithm}
\begin{algorithmic}[1]
\Require Start with $\yy^{(0)}$, $k=0$ , $\epscvg$, 	 a norm $\norm{\cdot}$ on $\rr^{NT}$ \;
\Repeat
 \State  $\xx^{(k+1)} \leftarrow P_{\X}( \yy^{(k)})$ \;
  \State  $\yy^{(k +1)} \leftarrow P_{\Y}(\xx^{(k+1)})$ \;
\State  $k \leftarrow k+1$ \;
\Until{$\norm{\xx^{(k)} -\xx^{(k-1)} } < \epscvg $}
\end{algorithmic}
\caption{Alternate Projections Method (APM)}
\label{algo:vonNeumannProj}
\end{algorithm}
The idea of using  cyclic projections to compute a point in the intersection of two sets comes from von Neumann \cite{von1950functional}, who applied it to the case of affine subspaces.
We next recall the basic convergence result concerning the APM. 
\begin{theorem}[\cite{GUBIN19671}] \label{thm:cvgAP}Let $\X$ and $\Y$ be two closed convex sets with $\X$ bounded, and let  $(\xE\kexp)_k $ and $(\xP\kexp)_k$ be the two infinite sequences generated by the APM on $\X$ and $\Y$ (\Cref{algo:vonNeumannProj}) with $\epscvg=0$. Then there exists $\xx^\infty \in \X$ and $\yy^\infty\in \Y $ such that:
\begin{subequations}
\begin{align}
& \xE\kexp \underset{k\rightarrow \infty}{\longrightarrow} \xx^\infty \ , \quad \xP\kexp \underset{k\rightarrow \infty}{\longrightarrow} \yy^\infty ; \\
& \norm{\xx^\infty - \yy^\infty }_2= \min_{{\xx \in \X , \yy \in \Y } } \norm{ \xE - \xP}_2 \ . 
\end{align}
\end{subequations}
In particular,  if  $\X \cap \Y \neq \emptyset$, then $(\xE\kexp)_k $ and $(\xP\kexp)_k$ converge to a same point $\xx^\infty \in \X \cap \Y$.

\end{theorem}
The convergence theorem is illustrated in \Cref{fig:cvgAPM} in the case where $\X \cap \Y= \emptyset$, that is, when the disaggregation problem \eqref{pb:disag} is not feasible.
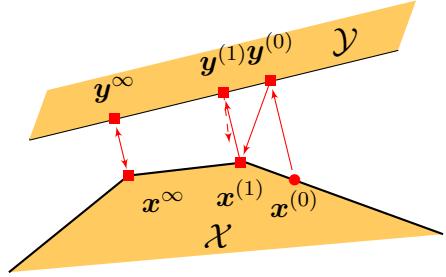
\begin{figure}[ht!]
\vspace{-0.5cm}
\begin{center}
\begin{tikzpicture}[rotate=-20,scale=0.7]

\draw [thick,draw= black, name path=lineY] (-2,-2) --   (4,2);
\path [ fill={rgb:orange,1;yellow,2;pink,3}] (-2,-1) -- (-2,-2) -- (4,2) -- (4,3) ;
\node[scale=1.2 ] at (3,1.8) {$\Y$} ;

\path [draw=black, fill={rgb:orange,1;yellow,2;pink,3}, thick,name path=lineX] (-1.5,-4.5) -- (0,-2) -- (2,-1) --   (6,-1);
\node [scale=1.2]at (2,-2.5) {$\X$};

\node [scale=2,fill=red, inner sep=1pt,label=-60:$\xx^\infty$]  (xinf) at (0,-2) {};

\draw[name path=line 1, opacity=0] (xinf) -- ($(xinf)+(-2,3)$);

\path [name intersections={of=lineY and line 1,by=yinf}];
\node [scale=2,fill=red, inner sep=1pt,label=90:$\yy^\infty$] at (yinf) {};
\draw [latex'-latex',red, shorten >=3pt] (xinf) -- (yinf);

\node (x0) at (3,-1) {};
\draw (x0) node [below  ]  {$\xx^{(0)}$} node [red] {$\bullet$};

\draw[name path=linex0y0, opacity=0] (x0) -- ($(x0)+(-2,3)$);
\path [name intersections={of=lineY and linex0y0,by=y0}];
\node [scale=2,fill=red, inner sep=1pt,label=90:$\yy^{(0)}$] at (y0) {};
\draw [-latex',red, shorten >=3pt] (x0) -- (y0);

\draw[name path=liney0x1, opacity=0] (y0) -- ($(y0)+(0,-3)$);
\path [name intersections={of=lineX and liney0x1,by=x1}];
\node [scale=2,fill=red, inner sep=1pt,label=-90:$\xx^{(1)}$] at (x1) {};
\draw [-latex',red, shorten >=3pt] (y0) -- (x1);

\draw[name path=linex1y1, opacity=0] (x1) -- ($(x1)+(-2,3)$);
\path [name intersections={of=lineY and linex1y1,by=y1}];
\node [scale=2,fill=red, inner sep=1pt,label=90:$\yy^{(1)}$] at (y1) {};
\draw [-latex',red, shorten >=3pt] (x1) -- (y1);

\draw [-latex',red,dashed, shorten >=3pt] (y1) -- ($(y1)+(0.5,-1)$);

\end{tikzpicture}
\end{center}
\caption{Alternate projections method (APM) on two sets $\X$ and $\Y$. \textit{When $\X \cap \Y = \emptyset$, APM cycles over two points $\xx^\infty$ and $\yy^\infty$.}\vspace{-0.6cm}}
\label{fig:cvgAPM}
\end{figure}
The idea of the algorithm proposed in this paper is, when $\Y_{\pp^{(s)}}\cap \X = \emptyset$, to use the resulting vectors $\xEinf $ and $\xPinf$ to construct a new subset $\P^{(s+1)}$ by adding a constraint of type \eqref{eq:hoffmanCondReverse} to $\P^{(s)}$: indeed, from \Cref{prop:hoffmanDisag}, we know that, if $\Y_{\pp^{(s)}}\cap \X = \emptyset$, there exists at least one violated inequality~\eqref{eq:hoffmanCondReverse}.

The difficulty is to guess such a violated inequality among the $2^{T}$ possible inequalities. It turns out that, using the output of APM, we can build such an inequality.

Suppose that we obtain $\xx^\infty \neq \yy^\infty $ as defined  in \Cref{thm:cvgAP}: we get a periodic cycle of APM, that is, we have $\xx^\infty= P_\X(\yy^\infty)$ and $\yy^\infty= P_\Y(\xx^\infty)$, and the couple $(\xx^\infty, \yy^\infty)$  is the solution of the following optimization  problem, with given parameters $(E_n)_n$ and $(p_t)_t$.
\begin{subequations} \label{pb:minSquare}
\begin{align}
&\min_{\xx,\yy} \frac{1}{2} \norm{\xx-\yy}_2^2 \\
&\forall n \in \Nall, \ \sum_{t\in\Tall} {x}\nt= E_n & (\lambda_n) \label{eq:minSquare:consE}\\
&\forall n \in \Nall, \forall t \in \Tall, \ \lbx\nt \leq {x}\nt \leq \ubx\nt   & (\umu\nt, \omu\nt) \label{eq:minSquare:consBounds}\\
& \forall t \in \Tall, \ \sum_{n\in\Nall} {y}\nt = p_t  & (\nu_t) \ \label{eq:minSquare:consp},
\end{align}
\end{subequations}
where $\lambda_n \in \rr$, $\umu\nt, \omu\nt \in \rr_+$ and  $\nu_t \in \rr$ are the Lagrangian multipliers associated to the constraints \eqref{eq:minSquare:consE},\eqref{eq:minSquare:consBounds},\eqref{eq:minSquare:consp},
with the associated Lagrangian function:
\begin{equation*}
\begin{split}
\L(\xx,\yy,  \llam,\mmu,\nnu)&=  \frac{1}{2} \norm{\xx-\yy}_2^2 - \llam^\tr \left( \txt\sum_t {x}\nt- E_n \right)_n  
 - \bm{\umu}^\tr (\xx - \llbx)\\ & \qquad\qquad\qquad\qquad -  \bm{\omu}^\tr(\uubx - {\xx})   - \nnu^\tr \left(\txt\sum_n {\yy}_n -\pp\right)  .
\end{split}
\end{equation*}
The stationarity condition of the Lagrangian with respect to the variable $y\nt$ yields:
 \begin{equation} \label{eq:stationarityKKT:ynt}
\forall n\in\Nall, \forall t \in\Tall, \ \y\nt = \x\nt + \nu_t   \ . 
\end{equation}

 Let us consider the sets  $\Tcut\subset \Tall $ and $\Ncut \subset \Nall$  defined from the output $(\xx^\infty, \yy^\infty)$ of APM on $\X$ and $\Y_{\pp}$ as follows:
 \begin{subequations}
   \label{eq:H0N0withoutbounds}
   \begin{align} \Tcut &\eqDef \big\{ t \in \Tall \ |  \ \exists n \in \Nall,\  {\y^\infty\nt} > \ubx\nt \big\} \enspace ,
     \\ \Ncut &\eqDef \big\{ n \in \Nall\ | \  E_n - \txt\sum_{t\notin \Tcut} \lbx\nt - \sum_{t\in \Tcut} \ubx\nt <0 \big\}  \enspace .
     \end{align}
\end{subequations}
In \Cref{th:cutN0T0violated} below, we show that applying the inequality \eqref{eq:hoffmanCondReverse} with the sets $
 \Tcut$ and $\Ncut $  defined in \eqref{eq:H0N0withoutbounds}  defines  a valid  inequality for the disaggregation problem violated by the current allocation $\pp$. 

The intuition behind the definition of $\Tcut$ and $\Ncut$ in \eqref{eq:H0N0withoutbounds} is the following: $\Tcut$ is the subset of resources for which there is an over supply (which overcomes the upper bound for at least one agent). Once  $\Tcut$ is defined,  $\Ncut$ is the associated subset of $\Nall$  minimizing the right hand side of \eqref{eq:ineq2powTv1}. Indeed, \eqref{eq:ineq2powTv1} can be rewritten as: 
\begin{equation*}
  \sum_{t \in \Tcut} p_t  \leq  \min_{\Ncut  \subset \Nall} \left\{ \sum_{n \in \Ncut } \Big( E_n - %
  \sum_{t \notin \Tcut} \lbx\nt  -\sum_{t \in \Tcut}  \ubx\nt \Big)\right\}   + \sum_{t \in \Tcut, n \in \Nall}  \ubx\nt .
\end{equation*}

\noindent The \Cref{th:cutN0T0violated} below is the key result on which relies the algorithm proposed in this paper.
\begin{theorem}\label{th:cutN0T0violated}
Consider the sequence of iterates $(\xE^{(k)},\xP^{(k)})_{k\in\nit}$ generated by APM on $\X$ and $\Y_{\pp}$ (see \Cref{algo:vonNeumannProj}). Then one of the following holds:
\begin{enumerate}[itemindent=0.7cm,label=(\roman*)]
\item if $ \X \cap \Y_{\pp} \neq \emptyset$, then $\xE^{(k)}, \xP^{(k)}  \underset{k \rightarrow \infty }{\longrightarrow} \xxinf \in \X \cap \Y_{\pp} $;
\item otherwise, if $ \X \cap \Y_{\pp} = \emptyset$, then $\xE^{(k)} \underset{k \rightarrow \infty }{\longrightarrow} \xxinf \in \X$ and  $\xP^{(k)} \underset{k \rightarrow \infty }{\longrightarrow} \xPinf \in \Y_{\pp}$. Then, considering the sets 
$\Tcut$ and $\Ncut$ defined in~\eqref{eq:H0N0withoutbounds}
gives an inequality of Hoffman~\eqref{eq:hoffmanCondReverse}  violated by $\pp$, that is:
\begin{equation}\label{eq:cutViolatedH0N0asym}
\sum_{n\in\Ncut} E_n - \sum_{t\in \Tcut} p_t + \sum_{t \in \Tcut, n \notin \Ncut} \ubx\nt  - \sum_{t \notin \Tcut, n \in \Ncut} \lbx\nt  < 0 \ .
\end{equation}
Moreover, this Hoffman inequality can be written as a function of $\xEinf$,  as follows:
\begin{equation}  \label{eq:cutConfidWithoutN0}
A_{\Tcut}(\xEinf)   < \sum_{t\in \Tcut} p_t \text{\  with }  A_{\Tcut}(\xEinf)\eqDef \sum_{t\in \Tcut} \sum_{n\in\Nall} \x^\infty\nt \ .
\end{equation}
\end{enumerate}
\end{theorem}
Before giving the proof of \Cref{th:cutN0T0violated}, we need to show some technical properties on the sets $\Tcut$, $ \Ncut$. For simplicity of notation,
 we use $\xx$ and $\yy$ to denote $\xxinf$ and $\yyinf$ in \Cref{prop:factsOncut} and the proof of \Cref{th:cutN0T0violated}.
\begin{proposition}  \label{prop:factsOncut} With $\xx \neq  \yy$ solutions of problem \eqref{pb:minSquare} (outputs of APM on $\X$ and $\Y_{\pp}$ with $\epscvg = 0$),
\begin{enumerate}[itemindent=0.7cm,label=(\roman*)]
\item \label{claimFactcut1} $\forall t \in \Tcut, \forall n \notin \Ncut, \ {x}\nt = \ubx\nt$ and $ {y}\nt > \ubx\nt$; 
\item \label{claimFactcut2} $\Tcut = \{t \ | \ \nu_t >0 \} = \{t \ | \ p_t > \sum_{n} x\nt \} $, where $\nu_t$ is the optimal Lagrangian multiplier associated to \eqref{eq:minSquare:consp};
\item  \label{claimFactcut4} $\forall n \in \Ncut, \lambda_n  < 0$ ;
\item \label{claimFactcut3}$\forall t \notin \Tcut, \forall n \in \Ncut, \ {x}\nt =\lbx\nt$ \ ;
\item \label{claimFactcutNonemptysets} the sets $\Tcut, \Tcut^c$, $\Ncut$ and $\Ncut^c$ are nonempty.
\end{enumerate}
\end{proposition}
The proof of \Cref{prop:factsOncut} is technical and given in \Cref{app:proofFactsonCuts}.
With \Cref{prop:factsOncut}, we are now ready to prove  \Cref{th:cutN0T0violated}. 
\begin{proof}[Proof of \Cref{th:cutN0T0violated}]
We have:
\begin{align*}
& \sum_{n \in \Ncut} E_n + \sum_{t \in \Tcut, n \notin \Ncut} \ubx\nt- \sum_{t \notin \Tcut, n \in \Ncut} \lbx\nt  -\sum_{t \in \Tcut} p_t \\
& =  \sum_{n \in \Ncut} \sum_{t} x\nt + \sum_{t \in \Tcut, n \notin \Ncut} \ubx\nt -\sum_{t \notin \Tcut, n \in \Ncut} \lbx\nt - \sum_{t \in \Tcut} p_t  \qquad  (\text{from } \eqref{eq:minSquare:consE}  ) \\ 
& = \sum_{n \in \Ncut}  \Big( \sum_{t\notin \Tcut} \lbx\nt + \sum_{t\in \Tcut} x\nt \Big) + \hh\hh \sum_{t \in \Tcut, n \notin \Ncut}\hh\hh x\nt- \hh\hh \sum_{t \notin \Tcut, n \in \Ncut} \hh\hh \lbx\nt - \sum_{t \in \Tcut} p_t 	\; \text{ (from  Prop.~\ref{prop:factsOncut} \ref{claimFactcut1} and \ref{claimFactcut3})} \\
&=  \sum_{t \in \Tcut} \Big( \sum_{n \in \Ncut} x\nt + \sum_{n\notin\Ncut} x\nt \Big) - \sum_{t \in \Tcut} p_t  = \sum_{t \in \Tcut} \Big(  \sum_{n\in\Nall} x\nt - \sum_{n\in\Nall} y\nt \Big) \\ 
& =\sum_{t \in \Tcut} \Big(  \sum_{n\in\Nall} -\nu_t \Big) =\sum_{t \in \Tcut} \Big(  - \sum_{n\in\Nall} \left| x\nt - y\nt \right|\Big)  
\end{align*}
using the stationarity conditions \eqref{eq:stationarityKKT:ynt} and for all $ t\in\Tcut, \nu_t >0$ by Prop.\ref{prop:factsOncut} \ref{claimFactcut2}. 
Moreover, using:
\begin{equation} \label{eq:sumzero}
\sum_{t \in \Tall} \sum_{n\in\Nall} ( x\nt - y\nt) = \sum_{n\in\Nall} E_n - \sum_{t\in \Tall} p_t =0 \ , \vspace{-0.2cm}
\end{equation}
obtained from \Cref{assp:PverifAggregCond}, we see that:
 \begin{equation} \label{eq:sumNorm1}
\sum_{t \in \Tcut} \Big(  - \sum_{n\in\Nall} \left| x\nt - y\nt \right|\Big)   = -({\norm{\xx-\yy}_1} )/ {2} < 0 \ , 
\end{equation}
which shows \eqref{eq:cutViolatedH0N0asym}.
We now show  that  inequality \eqref{eq:cutConfidWithoutN0} is obtained by a rewriting of  \eqref{eq:cutViolatedH0N0asym}, indeed:
\begin{align*}
 & \sum_{n\in\Ncut} E_n +  \sum_{t \in \Tcut, n \notin \Ncut} \ubx\nt  - \sum_{t \notin \Tcut, n \in \Ncut} \lbx\nt  \\
&=  \sum_{n\in\Ncut} \sum_{t\in \Tall} x\nt + \sum_{t \in \Tcut, n \notin \Ncut} x\nt  - \sum_{t \notin \Tcut, n \in \Ncut} x\nt  \qquad  \text{ (from  Prop.\ref{prop:factsOncut} \ref{claimFactcut1} and \ref{claimFactcut3})} \\
&=  \sum_{t\in \Tcut, n\in\Ncut} x\nt + \sum_{t \in \Tcut, n \notin \Ncut} x\nt   = \sum_{t\in \Tcut} \sum_{n\in\Nall} x\nt \ =A_{\Tcut}(\xE) .
\end{align*}
\end{proof}

\begin{remark}
  An alternative to APM is to use Dykstra's projections algorithm \cite{dykstra1983algorithm}.
  It is shown in \cite{bauschke1994dykstra} that the outputs of Dykstra's algorithm also satisfy the conditions given in \Cref{thm:cvgAP}, thus \Cref{th:cutN0T0violated} will also hold for this algorithm.
  However, in this paper we focus on APM as it is simpler and, to our knowledge, there is no guarantee that Dykstra's algorithm would be faster in this framework.
\end{remark}

Suppose, as before, that the two sequences generated by APM on $\X$ and $\Y$ converge to two distinct points $\xxinf$ and $\yyinf$.
 Then, at each round $k$%
, we can define from \eqref{eq:stationarityKKT:ynt} and considering any $n\in\Nall$, the multiplier $\nnu\kexp=\yy_n\kexp-\xx_n\kexp $ tends to $ \yy_n^\infty-\xx^\infty_n \eqd \nnu^\infty $.
 The set $\Tcut$ of \Cref{th:cutN0T0violated} is:
 \begin{equation}
 \Tcut^\infty \eqd \{t \in\Tall \  | \  0 <  \nu_t^\infty \}, 
  \end{equation}
which raises an issue for practical computation, as $\nnu^\infty$ is only obtained \emph{ultimately} by APM, possibly in infinite time.
  To have access to $\Tcut^\infty$ \emph{in finite time}, that is, from one of the iterates $(\nnu^{(k)})_k$,  we consider the set: %
\begin{equation}
\Tcut^{(K)} \eqd \{t \in \Tall \ | \ \Bseuil \epscvg  < \nu_t^{(K)} \},
\end{equation}
where $\epscvg$ is the tolerance for convergence of APM as defined in \Cref{algo:vonNeumannProj}, $\Bseuil>0$ is a constant, and $K$ (depending on $\epscvg$) is the first integer such that $\norm{\xx^{(K)} - \xx^{(K-1)}} <\epscvg$. 

We next show that we can choose $\Bseuil$ to ensure that $\Tcut^{(K)}= \Tcut^{\infty}$  for $\epscvg$ small enough. 
We rely on the geometric convergence rate of APM on polyhedra  \cite{borwein2014analysis,nishihara2014convergence}:
\begin{proposition}[{\cite{nishihara2014convergence}}]
\label{prop:geometricCvgGeneral}
If $\X$ and $\Y$ are polyhedra, there exists %
$\rho \in (0,1)$ such that the sequences $(\xx\kexp)_k$ and $(\yy\kexp)_k $ generated by APM verify for all $ k\geq 1$:
\begin{align*}
&\norm{\xE^{(k+1)} - \xE^{(k)} }_2  \leq \rho \norm{\xE^{(k)} - \xE^{(k-1)} }_2  \quad \text{ and } \quad  \norm{\yy^{(k+1)} - \yy^{(k)} }_2  \leq \rho \norm{\yy^{(k)} - \yy^{(k-1)} }_2  \ .
\end{align*}
\end{proposition}  
\Cref{prop:geometricCvgGeneral} applies to any polyhedra $\X$ and $\Y$. In \Cref{subsec:complexityAnalysis} we shall give   %
  an explicit upper bound on the constant $\rho$ in the specific transportation case given by \eqref{cons:disagfeas} and  \eqref{eq:XnTransport}.

From the previous proposition, we can quantify the distance to the limits in terms of $\rho$:
\begin{lemma}
Consider an integer $K$ such that the sequence $(\xx\kexp)_{k \geq 0}$ generated by APM satisfies $\norm{\xx^{(K)}-\xx^{(K-1)}}\leq \epscvg$,  %
 then we have for any $K'\geq K-1$:
\begin{equation*}
 \norm{\xEinf - \xE^{(K')}} \leq \frac{\epscvg}{1-\rho} \ .
\end{equation*}
\end{lemma}
\begin{proof} From \Cref{prop:geometricCvgGeneral}, we have for any $k \geq K$: 
\begin{align*}
 & \norm{\xE^{(k)} - \xE^{(K')} } \leq \sum_{s=0}^{k-K'} \norm{\xE^{(K+s+1)} - \xE^{(K+s)} } %
 \leq    \sum_{s=0}^{k-K'} \rho^{s} \norm{\xE^{(K'+1)} - \xE^{(K')} } \leq \frac{1}{1-\rho} \epscvg \ ,
\end{align*}
so that, by taking the limit $k \rightarrow \infty$, one obtains $ \norm{\xEinf - \xE^{(K')}} \leq \tfrac{\epscvg}{1-\rho}$.
\end{proof}
With this previous lemma, we can state the condition on $\Bseuil$ ensuring the desired property:
\begin{proposition} \label{prop:convTinf}
Define $\lnu \eqDef \min \{ |\nu^\infty_t|>0\}$ (least nonzero element of $\nnu^\infty$). If the constants $\Bseuil$ and $\epscvg>0$ are chosen such that %
$ \Bseuil > \frac{1}{1-\rho} \text{  and } \  \epscvg \times  2 \Bseuil  < \lnu , $
and \Cref{algo:vonNeumannProj} stops at iteration $K$, then we have:
 $$\Tcut^{(K)}= \Tcut^\infty\ . $$
\end{proposition}

\begin{proof}
Let  $t \in \Tcut^\infty$, that is $ \nu_t^\infty >0 $ which is equivalent to $ \nu_t^\infty \geq  \lnu$ by definition of $\lnu$. We have:
\begin{align*}
\nu_t^{(K)} = \frac{1}{N}(p_t - \sum_n x\nt^{(K)})
& = \frac{1}{N}(p_t - \sum_n x^\infty\nt) + \frac{1}{N}(\sum_n x^\infty\nt - \sum_n x\nt^{(K)}) \\
& >  \nu_t^\infty - \tfrac{\epscvg}{1-\rho} \geq  \lnu - \tfrac{\epscvg}{1-\rho} > \epscvg (2 \Bseuil  - \tfrac{1}{1-\rho})\ ,
\end{align*}
and this last quantity is greater than $\Bseuil \epscvg$ as soon as
 $\Bseuil\geq \tfrac{1}{1-\rho}$, thus $t \in \Tcut^{(K)}$. 

Conversely, if $t \in \Tcut^{(K)}$, then:
\begin{align*}
 \nu_t^\infty  &=\frac{1}{N}(p_t - \sum_n x^\infty\nt)  =\frac{1}{N}(p_t - \sum_n x\nt^{(K)})    - \frac{1}{N}(\sum_n x^\infty\nt - \sum_n x\nt^{(K)})  \\
& \geq \nu_t^{(K)} - \tfrac{\Bseuil}{1-\rho } \ > \ \nu_t^{(K)} - \Bseuil \epscvg \geq (\Bseuil-\Bseuil) \epscvg \geq 0 \ ,
\end{align*}
so that $t \in  \Tcut^\infty$.
Furthermore, the ``approximated'' cut $\sum_{t\in \Tcut} \left( \sum_{n\in\Nall} \x^{(K)}\nt  - p_t \right) \geq 0 $ is violated by the current value of $\pp$  (or $\pp^{(s)}$ at iteration $s$) in the algorithm as:
\begin{align*}
\sum_{t\in \Tcut} \left( \sum_{n\in\Nall} \x^{(K)}\nt  - p_t \right)
\leq & \sum_{t\in \Tcut} \left( \sum_{n\in\Nall} \x^{(K)}\nt - x^\infty\nt \right) + \sum_{t\in \Tcut} \left( \sum_{n\in\Nall}  \x^\infty\nt- p_t \right) \\
& \leq \norm{\xE^{(K)}-\xEinf}_1 -\frac{1}{2}\norm{\xEinf-\xPinf}_1
\end{align*}
using \eqref{eq:sumzero} and \eqref{eq:sumNorm1}. This last quantity is negative as soon as $\norm{\xE^{(K)}-\xEinf}_1 < \frac{1}{2}\norm{\xEinf-\xPinf}_1$, which holds in particular if $\Bseuil \epscvg < \frac{1}{2}\norm{\xEinf-\xPinf}_1$.
\end{proof}

This second proposition shows a surprising result:  even if we do not have access to the limit $\xxinf$, we can compute \emph{in finite time} the \emph{exact} left hand side term $\Arhs_{\Tcut}(\xxinf)$ of the cut  \eqref{eq:cutConfidWithoutN0}:
\begin{proposition} \label{prop:exactnessAT0K}
Under the hypotheses of \Cref{prop:convTinf}, we have: 
\begin{equation*}
\Arhs_{\Tcut}(\xx\Kexp) = \sum_{t\in\Tcut} \sum_{n\in\Nall} \x\nt\Kexp = \Arhs_{\Tcut}(\xxinf) \ .
\end{equation*}
\end{proposition}
\begin{proof}
We start by showing some technical properties similar to \Cref{prop:factsOncut}:
\begin{lemma} \label{lemm:factsCutFinite}
The iterate $\xx\Kexp$ satisfies the following properties:
\begin{enumerate}[label=(\roman*),wide]
\item \label{claimFactcutFinite1} $\forall t \in \Tcut, \forall n \notin \Ncut,  {x}\nt\Kexp =\x^\infty\nt= \ubx\nt$ ; 
\item \label{claimFactcutFinite2}$\forall t \notin \Tcut, \forall n \in \Ncut, \ {x}\nt\Kexp=\x^\infty\nt =\lbx\nt$ \ .
\end{enumerate}
\end{lemma}
The proof of \Cref{lemm:factsCutFinite} is similar to \Cref{prop:factsOncut} and is given in \Cref{app:proofCutFinite}. Then, having in mind that $\Tcut\Kexp= \Tcut^\infty$ from \Cref{prop:convTinf}, and $\Ncut$ is obtained from $\Tcut^\infty$  by \eqref{eq:H0N0withoutbounds}, we obtain:
\begin{align*}
\Arhs_{\Tcut}(\xx\Kexp) &= \sum_{n\in\Ncut} \Big( \sum_{t\notin \Tcut} \lbx\nt + \sum_{t\in \Tcut} \x\nt\Kexp \Big)  - \sum_{t\notin \Tcut, n \in \Ncut} \lbx\nt + \sum_{t\in \Tcut, n \notin \Ncut} \x\nt\Kexp \\
& \leq \sum_{n\in\Ncut} \sum_{t\in\Tall} \x\nt\Kexp - \sum_{t\notin \Tcut, n \in \Ncut} \lbx\nt + \sum_{t\in \Tcut, n \notin \Ncut} \ubx\nt \quad \text{ (from \Cref{lemm:factsCutFinite}) } 
\end{align*}
which equals to $\Arhs_{\Tcut}(\xxinf)$  as we have $\sum_{t\in\Tall} \x\nt\Kexp=E_n$ for each $\n\in\Nall$.
\end{proof}

Before presenting our algorithm using this last result, we focus on the technique of secure multiparty computation (SMC) which will be used here to ensure the privacy of agent's constraints and profiles while running APM.%

\subsection{Privacy-Preserving Projections through SMC}
\label{subsec:APMwithSMC}
APM, as described in  \Cref{algo:vonNeumannProj}, enables a distributed implementation in our context, by the structure of the algorithm itself: the operator computes the projection on $\Y_{\pp}$ while each agent $n$ can compute, possibly in parallel, the projection on $\X_n$ of the new profile transmitted by the operator. 
This enables each agent (as well as the operator) to keep her individual constraint and not reveal it to the operator or other agents. 
However, if this agent had to transmit back her newly computed individual profile to the operator for the next iteration, privacy would not be preserved.
We next show that, using a  secure multi-party computation (SMC) summation protocol~\cite{yao1986secrets} we can avoid this communication of individual profiles and
implement APM without revealing the sequence of agents profiles $\xE$ to the operator. 

For this, we use the fact that $\Ypp$ is an affine subspace and thus the projection on  $\Ypp$ can be obtained explicitly component-wise. Indeed, as $\P_{\Ypp}(\xx)$ is the solution $\yy$ of the quadratic program $\min_{\yy \in \rr^{NT}  | \sum_n \yy_n= \pp } \tfrac{1}{2} \|\xx-\yy\|_2^2$, we obtain optimality conditions similar to \eqref{eq:stationarityKKT:ynt}, and summing these equalities on $\Nall$, we get $N \nu_t = \pp_t- \sum_{n\in\Nall} \x\nt$ for each $t\in\Tall$, and thus:
\begin{equation} \label{eq:explicitProjY}
\forall n \in\Nall, \ [\P_{\Ypp}(\xx)]_n = \xx_n + \tfrac{1}{N}(\pp-\txt\sum_{m\in\Nall} \xx_m) \ .
\end{equation}
Thus, having access to the \emph{aggregate} profile $\SS \eqd \sum_{n\in\Nall} \xx_n$, each agent can compute locally the component  of the projection on $\Ypp$ of her profile, instead of transmitting the profile to the operator for computing the projection in a centralized way.

Using SMC principle, introduced by \cite{yao1986secrets} and \cite{goldreich2019howto},  the sum $\SS$ can be computed in a non-intrusive manner and by several communications between agents and the operator, as described in \Cref{algo:SMCsum}.
The main idea of the SMC summation protocol \cite{atallah2004private,shi2016secure} below is that, instead of sending her profile $\xE_n$,  agent  $n$ \emph{splits} $\x\nt$  for each $t$ into $N$ random parts $(s_{n,t,m})_m$, according to an uniform distribution and summing to $\xE\nt$ (Lines \ref{algline:SMC:cut1}-\ref{algline:SMC:cut2}). Thus, each part $s_{n,t,m}$ taken individually does not reveal any information on $ \xE_n$ nor on $\X_n$, and can be sent to agent $m$. 
Once all exchanges of parts are completed (Line 5), and $n$ has herself received the parts from other agents, agent $n$ computes a new aggregate quantity $\ssig_n$ (Line \ref{algline:NIAPM-computesigma}), which does not contain either any information about any of the agents, and sends it to the operator (Line \ref{algline:NIAPM-Sendsigma}). The operator can finally compute the quantity $\SS=\xE^\tr \dsoneN= \ssig^\tr \dsoneN$. 

\begin{algorithm}[!ht]
\begin{algorithmic}[1]
 \Require  profile $\xx_n$ for each  ${n\in \Nall} $, parameter $R> \max_t \max\{\sum_n \x\nt\} $ given to each  agent \;
  
	\For{each agent $\n \in \Nall$}
  \State Draw $\forall t, (s_{n,t,m})_{m=1}^{N-1} \hh \in\hh \mathcal{U}([0,R]^{N-1})$ \; \label{algline:SMC:cut1}
  \State and set $\forall t,  s_{n,t,N}\hh \eqd \hh {\x}_{n,t}- \sum_{m=1}^{N-1} s_{n,t,m} \mathrm{mod}\ R$\; \label{algline:SMC:cut2}
  \State Send $(s_{n,t,m})_{t\in\Tall} $ to agent $m \in \Nall$ \;
  \EndFor
  \For{each agent $\n \in \Nall$}
  \State Compute $\forall t, \sigma_{n,t} = \sum_{m\in\Nall} s_{m,t,n} \mathrm{mod} \ R $  \label{algline:NIAPM-computesigma}\;
  \State Send $(\sigma_{n, t})_{t\in\Tall}$ to operator \label{algline:NIAPM-Sendsigma}\;
  \EndFor \label{algline:NIAPM-SendsigmaOp} 
  \State Operator computes $\SS= \sum_{n\in\Nall} \ssig_n \mathrm{mod} \ R$ (and broadcasts it to agents)\;
\end{algorithmic}
\caption{SMC of Aggregate (SMCA) $\sum_ {n\in\Nall} \xx_n$}
\label{algo:SMCsum}
\end{algorithm}
\begin{remark}
  A drawback of the proposed SMC summation (\Cref{algo:SMCsum}) is that each agent need to exchange with all other nodes, therefore the communication overhead for each iteration and each node is $\mc{O}(N \ell)$ where $\ell$ is an upper bound on the length of the message (e.g number of bits encoding constant $R$).
Splitting the messages with the other $N-1$ nodes is required  to obtain privacy guarantees against collusion of strictly less than $N-1$ nodes (see \Cref{sec-guarantees}).
However, as detailed in  \cite[Protocol I]{atallah2004private}  the SMC summation can be adapted by splitting secret numbers among only $k \in \{1,\dots, N\}$ members. In that case, the communication overhead would be reduced to $\mc{O}(k \ell)$, but privacy would only be protected against collusion of less than $k-1$ agents. This choice has to be made as a tradeoff between computation overhead and privacy guarantees.
\end{remark}
We sum up in \Cref{algo:vonNeumannProjasymConfid} below the procedure of generating a new constraint as stated in \Cref{th:cutN0T0violated} from the output of APM in finite time (see \Cref{prop:convTinf})  and in a privacy-preserving way using SMC.

\begin{algorithm}[!ht]
\begin{algorithmic}[1]
\Require Start with $\yy^{(0)}$, $k\hh=1 $, $\epscvg,\epsdis$, norm $\norm{.}$ on $\rr^{NT}$ \;
\Repeat \label{algline:NIAPM-mainLoop}
 	\For{each agent $\n \in \Nall$}
  \State  ${\xx}_n^{(k)} \leftarrow P_{\X_n}( {\yy}_n^{(k-1)})$  \label{algline:NIAPM-compute-x}\;
  \EndFor
 \State Operator obtains $\SS^{(k)} \leftarrow $SMCA($\xx^{(k)}$)  (\textit{cf} Algo.~\ref{algo:SMCsum}) \label{algline:NIAPM-optain-agg}\;
\State  and sends $\nnu^{(k)}\hh\eqd\frac{1}{N}( \pp-\SS^{(k)} ) \in \rr^T$ to agents $\Nall$ \label{algline:NIAPM-opsendsnu}\;
  	\For{each agent $\n \in \Nall$}
  	\State Compute $\xP_\n^{(k)} \leftarrow \xE_\n^{(k)} + \nnu^{(k)} $ \label{algline:NIAPM-compute-y} \Lcomment{ from \eqref{eq:stationarityKKT:ynt} and \eqref{eq:explicitProjY},  $\yy\kexp= P_{\Y_{\pp}}(\xx\kexp)$ }\;
\EndFor
\State  $k \leftarrow k+1$ \;
\Until{ $ \norm{\xx^{(k)} - \xx^{(k-1)}  } < \epscvg$ } \label{algline:NIAPM-conv-APM}
\If { $\norm{\xx^{(k)} - \yy^{(k)} } \leq \epsdis$} \label{algline:NIAPM-conv-disag}
\Lcomment{found a $\epsdis$-solution of the disaggregation problem}
\State  Each agent adopts profile $\xE_\n^{(k)} $ \;	
\State \Return \textsc{Disag} $\leftarrow$ \textsc{True} \;
\Else 
\Lcomment{have to find a valid constraint violated by $\pp$} 
\State Operator computes $\Tcut \leftarrow \{t \in \Tall \ | \ \Bseuil \epscvg < \nu_t^{(k)}  \}$  \label{algline:NIAPM-defH0asymConfid}\;
\State Operator computes $\Arhs_{\Tcut} \leftarrow $ SMCA( $ (\xx_t^{(k)})_{t\in\Tcut} $)    \;
\If {$\Arhs_{\Tcut}- \sum_{t\in\Tcut} \pp_t <0$} \label{algline:NIAPM:checkconstraint}
\State \Return \textsc{Disag} $\leftarrow$ \textsc{False}, $\Tcut, \ \Arhs_{\Tcut}$ \label{algline:NIAPM-returncons}\;
\Else 
\Lcomment{need to run APM with higher precision}
\State Return to Line~\ref{algline:NIAPM-mainLoop} with $\epscvg \leftarrow \epscvg/2$ \label{algline:NIAPM-resolve}\;
\EndIf \;
\EndIf \;
\end{algorithmic}
\caption{Non-intrusive APM %
 (NI-APM)}
\label{algo:vonNeumannProjasymConfid}
\end{algorithm}

To choose $\Bseuil$ and $\epscvg$ satisfying the conditions of \Cref{prop:convTinf} \textit{a priori}, one has to know the value of $\lnu$. Although a conservative lower bound could be obtained by Diophantine arguments if we consider rationals as inputs of the algorithm, in practice it is easier and more efficient to proceed in an iterative manner for the value of $\epscvg$.  
Indeed, one can start with $\epscvg$ arbitrary large so that APM will converge quickly, and then check if the cut obtained is violated by the current value of $\pp$ (\Cref{algline:NIAPM:checkconstraint}): if it is not the case, we can continue the iterations of APM with convergence precision improved to $\epscvg/2$ (\Cref{algline:NIAPM-resolve}). \Cref{prop:convTinf} ensures that this loop terminates in finite time.

The parameter $\epsdis>0$ (Line \ref{algline:NIAPM-conv-disag} of \Cref{algo:vonNeumannProjasymConfid}) has to be chosen a priori by the operator, depending on the  precision required. In general in APM, $\xxinf=\yyinf$  will only be achieved in infinite time, so choosing $\epsdis$ strictly positive is required. 

\medskip

We end this section by summarizing in \Cref{algo:confidOptimDisag} the global iterative procedure to compute an optimal and disaggregable resource allocation $\pp$, solution of the initial problem \eqref{pb:global}, using iteratively NI-APM (\Cref{algo:vonNeumannProjasymConfid}) and adding constraints as stated in \Cref{th:cutN0T0violated}.

\begin{algorithm}[!ht]
\begin{algorithmic}[1]
\Require $s=0$ , $\P^{(0)}=\P$ ;
\dsg = \textsc{False} \;
\While{Not \dsg}
\State Solve $\min_{ \pp \in \P\ids } f(\pp)$ \label{algline:Optim:masterProb} \; 
 \If {problem infeasible}
 \State Exit  \label{algline:exitNoSolution}\;
 \Else 	\State Compute $\pp\ids= \arg\min_{ \pp \in \P\ids } f(\pp)$
	\EndIf\; 
 	\State \dsg $\leftarrow$ NI-APM($\pp\ids$)  \ (Algo.~\ref{algo:vonNeumannProjasymConfid})\;
	\If{ \dsg }  \label{algline:Optim:disagTrue}
	\State Operator adopts $\pp\ids $ \;
	\Else \label{algline:Optim:disagNot}
 \State Obtain $ \Tcut\ids, \Arhs_{\Tcut}\ids $ from NI-APM($\pp\ids$) \; 
	\State $\P^{(s+1)} \leftarrow \P\ids \cap \{ \pp | \sum_{t\in\Tcut\ids} p_t \leq  \Arhs_{\Tcut}\ids \}$ \label{algline:Optim:addConstraint}\;
 \EndIf \;
 \State $s \leftarrow s+1 $\;
\EndWhile  \;
\end{algorithmic}
\caption{Non-intrusive Optimal Disaggregation}
\label{algo:confidOptimDisag}
\end{algorithm}

\Cref{algo:confidOptimDisag} iteratively calls  NI-APM (\Cref{algo:vonNeumannProjasymConfid}) and in case disaggregation is not possible (Line \ref{algline:Optim:disagNot}), a new constraint is added (Line \ref{algline:Optim:addConstraint}), obtained from the quantity $\Arhs_{\Tcut}$ defined in \eqref{eq:cutConfidWithoutN0}, to the feasible set of resource allocations $\P^{(s)}$ in problem \eqref{pb:master}. 
This constraint is an inequality on $\pp$ and thus does not reveal significant individual information to the operator. 
The algorithm stops when disaggregation is possible (Line \ref{algline:Optim:disagTrue}). 
The termination of \Cref{algo:confidOptimDisag} is ensured by the following property and the form of the constraints added \eqref{eq:cutViolatedH0N0asym}:
\begin{proposition} \label{prop:finiteNumberCuts}
\Cref{algo:confidOptimDisag} stops after a finite number of iterations, as at most $2^T-2$ constraints (Line \ref{algline:Optim:addConstraint}) can be added to the master problem (Line \ref{algline:Optim:masterProb}).
\end{proposition}
The following \Cref{prop:correctnessAlgo} shows the correctness of our \Cref{algo:confidOptimDisag}.
\begin{proposition}\label{prop:correctnessAlgo}
Let $\Bseuil$ and $\epscvg$ satisfy the conditions of \Cref{prop:convTinf} and \ref{prop:exactnessAT0K}. Then:
\begin{enumerate}[wide,label=(\roman*)]
\item if the problem \eqref{pb:global} has no solution,  \Cref{algo:confidOptimDisag} exits at Line \ref{algline:exitNoSolution} after at most $2^T-2$ iterations;
\item otherwise, \Cref{algo:confidOptimDisag}   computes,  after at most $s \leq 2^T-2$ iterations, an aggregate solution $\pp\ids \in \P$, associated to individual profiles $(\xx^*)_n = $ NI-APM($\pp\ids$) such that:
\begin{equation*}
\pp\ids \in \P, \quad, \forall n\in\Nall, \xx_n ^*\in \X_n, \quad \big\|\txt\sum_{n\in\Nall} \xx_n^* - \pp\ids \big\| \leq \epsdis \quad  \text{ and } f(\pp\ids) \leq f^* \ ,
\end{equation*}
 where $f^*$ is the optimal value of problem \eqref{pb:global}. 
\end{enumerate}
\end{proposition}
\begin{proof}
The proof is immediate from \Cref{th:cutN0T0violated}, \Cref{prop:convTinf} and \Cref{prop:exactnessAT0K}.
\end{proof}

\begin{remark}
  The upper bound on the number of constraints added has no dependence on $N$ because, as stated in \eqref{eq:ineq2powTv1}, once a subset of $\Tall$ is chosen, the constraint we add in the algorithm is found by taking the minimum over the subsets of $\Nall$.
\end{remark}
Although there exist some instances with an exponential number of independent constraints, this does not jeopardize the proposed method: %
 in practice, the algorithm stops after a very small number of constraints added.
  Intuitively, we will only add constraints ``supporting'' the optimal allocation $\pp$.
Thus, \Cref{algo:confidOptimDisag} is a method which enables the operator to compute a resource allocation $\pp$ and the $N$  agents to adopt profiles $(\xx_n)_n$, such that $(\xx,\pp)$ solves the global problem \eqref{pb:global}, and the method ensures that both 
 agent constraints (upper bounds $\uubx_n$, lower bounds $\llbx_n$, demand $E_n$); 
and disaggregate (individual) profile $\xx_n $ (as well as the iterates $(\xE^{(k)})_k$ and $(\xP^{(k)})_k$ in NI-APM)
are kept confidential by agent $n$ and can not be induced by a third party (either the operator or any other agent $m\neq n$).

\begin{remark} %
\label{rem:SMCLagrangian} A natural approach to address problem \eqref{pb:global} in a distributed way, assuming that both the cost function $\pp \mapsto f(\pp)$ and the feasibility set $\P$ are convex, is to rely on Lagrangian based decomposition techniques. Examples of such methods are Dual subgradient methods \cite[Chapter 6]{bertsekas1997nonlinear}, auxiliary problem principle method \cite{cohen84decomposition}, ADMM \cite{glowinski1975approximation},\cite{yu2017simple} or bundle methods \cite{lemarechal1995new}. 

It is conceivable to develop a privacy-preserving implementation of those techniques, where Lagrangian  multipliers associated to the (relaxed) aggregation constraint $\sum_n \xx_n = \pp$ would be updated using the  SMC technique as described in \Cref{algo:SMCsum}. 
However, those techniques usually ask for strong convexity hypothesis: for instance, in ADMM, in order to keep the decomposition structure in agent by agent, a possibility is to use multi-blocs ADMM  with $N+1$ blocs ($N$ agents and the operator), which is known to converge in the condition that strong convexity of the cost function  in at least $N$ of the $N+1$ variables holds  \cite{deng2017parallel}. 
The study of privacy-preserving implementations of Lagrangian decomposition methods is left for further work.

The advantage of  \Cref{algo:confidOptimDisag} proposed in this paper is that convergence is ensured (see \Cref{prop:finiteNumberCuts}) even if the cost function $\pp \mapsto f(\pp)$ and the feasibility set $\P$ are not convex, which is the case in many practical situations (see \Cref{subsec:appli-microgrid}).
\end{remark}
\subsection{Privacy guarantees}\label{sec-guarantees}
We next analyze the protection of privacy of the present optimization
algorithm, considering different situations.
For this analysis, we consider the paradigm of \emph{honest, semi-honest} and \emph{malicious} agents.
The term {\em semi-honest} (also called \emph{honest but curious}  \cite{abbe2012privacy},\cite{ruan2019secure}) refers to a user who obeys
the protocol, but wishes to use the data obtained through the algorithm
to infer private information.
In contrast, a {\em malicious}
user may even disobey the protocol.

\paragraph{Malicious operator and honest agents}
As a first step, we consider an ideal model, in which we only
analyze the loss of privacy caused by the algorithm,
through communications
between honest agents and a malicious (or semi-honest) central operator.

We neglect, in this situation, the information leaks induced
by inter-agent communications through
the secure multiparty protocol.

Through \Cref{algo:confidOptimDisag}, the operator obtains the following information  for each iteration $s \in \{1,\dots, \scvg\} $ (where $\scvg<\infty$ is the number of iterations needed for convergence):
\begin{itemize}
\item the set   $\Tcut\ids$ and the scalar $\Arhs_{\Tcut}\ids$ defining the cut;
\item whenever \textsc{ni-apm} is called, the serie $(S^{(s,k)})_{k=1}^{K_s}$ of the aggregated projections.
\end{itemize}
Let us denote by $\mathbb{I}^{\text{op}}\eqd\big\{(S^{(0,k)})_{k=1}^{K_0}, \Tcut^{(0)}, \Arhs_{\Tcut}^{(0)}, \dots, (S^{(\scvg,k)})_{k=1}^{K_{\scvg}}, \Tcut^{(\scvg)}, \Arhs_{\Tcut}^{(\scvg)} \big\} $ the information obtained by the operator.
The next proposition implies that, in this ideal model,
the operator cannot discern the individual profiles or individual constraints.

\begin{proposition}\label{prop-safe}
  Consider two different instances of the resource allocation
  problem, differing only by a permutation of the agents,
  meaning that for all $n\in \Nall$, $\mathcal{X}_n$ is replaced
  by $\mathcal{X}_{\sigma(n)}$, for some permutation $\sigma$
  of the elements of $\Nall$. Then, the information $\mathbb{I}^{\text{op}}$ gotten
  by the operator (as well as the information $(\nnu^{(k)})_k$ sent by the operator to agents in \Cref{algo:vonNeumannProjasymConfid}), and the sequence of cuts added
  in the master problem (\Cref{algo:confidOptimDisag}),
  are precisely the same in both instances.
\end{proposition}
\begin{proof}
  The only information transmitted by the agents
  to the central operator consists of  the aggregate profiles $\SS^{(k)} = $ SMCA$(\xx^{(k)})$  and  the sums $A_{\Tcut}$
  (line~5 and line~17 of~\Cref{algo:vonNeumannProjasymConfid}),
  for different values of the subset $\Tcut$, corresponding
  to different Hoffman cuts.
  Each aggregate profile $\SS^{(k)}$ is invariant by permutation of the agents. 
  Each sum $A_{\Tcut}$ is of the form $A_{\Tcut}=\sum_{n\in \Nall,t\in \Tcut}x_{nt}^{(K)}$, where
  $K$ is the iteration at which the violated inequality is found. Hence
  it is also invariant by a permutation of the agents.
\end{proof}
To illustrate \Cref{prop-safe}, suppose Alice and Bob
live in a small town, in which there is a sports Hall,
with two kind of lessons every day,
boxing, from 12h00 to 13h00, and dancing, from 18h00 to 19h00.
Suppose Alice goes to the boxing lesson, and that Bob goes to the dancing
lesson. Suppose further that Alice and Bob do not consume
any electric energy during their lessons. Then, by analyzing the successive aggregates $A_{\Tcut}$, for different values of $\Tcut\subset\Tall$,
the central operator may deduce that the global consumption
decreases at the time of the lessons, implying that
one agent boxes whereas another one dances.
However, whether it is Alice or Bob who does boxing or dancing
remains inaccessible to the operator, owing to the symmetry argument.
This protection persists even with a malicious operator. This privacy property is stated formally as follows:
\begin{corollary}\label{cor:safeOperator}
With $N>1$, a malicious (or semi-honest) operator using \Cref{algo:confidOptimDisag} cannot infer the constraints $\X_n$ or profile $\xx_n$ of a particular agent $n \in \Nall$ with probability 1.
  \end{corollary}

  \begin{remark} \label{rm:info-leak}
    Even if the identity of users is fully protected,  in some specific cases the central operator could infer limited information about  constraints or profile that an agent within the population may bear.
    Indeed, let us consider an illustrative example with  $T=2$, $N=2$, and constraints parameters $(E_1,\ubx_{1,1},\ubx_{1,2}) \eqd (1,2,0)$, $(E_2,\ubx_{2,1},\ubx_{2,2}) \eqd (3,1,2)$  and $\llbx \eqd 0$ (all unknown to the operator).
    The operator knows \textit{a priori} the aggregate quantities  $E_1+E_2=4, \  \uubx_1+\uubx_2=(3,2),\  \llbx_1+\llbx_2=(0,0)$.

    Before the algorithm, from the operator viewpoint, any parameter values satisfying the aggregate conditions above are possible.
   For instance, it is possible that agent 1 has the parameters $(E^*_1,\ubx^*_{1,1},\ubx^*_{1,2},\lbx^*_{1,1},\lbx^*_{1,2}) \eqd (3,2,2,0,0)$.
   
   Now suppose that the first aggregate profile proposed by the operator is $\pp^{(0)}\eqd (3,1)$. 
   From \Cref{algo:confidOptimDisag}, he obtains the cut defined by $\T^{(0)}=\{1\}$ and $A_{\{1\}}^{(0)}=2$.
   Thus, assuming the operator knows that $A_{\{1\}}^{(0)}$ is of the form, $A_{\{1\}}^{(0)}= \min_{\N \subset \Nall}\{ \sum_{n\in\Ncut} E_n + \sum_{t \in \Tcut, n \notin \Ncut} \ubx\nt  - \sum_{t \notin \Tcut, n \in \Ncut} \lbx\nt  \}$ for a subset $\N$, he infers the following conditions on the parameters:
   \begin{align} \label{eq:cond-op-privacy}
     \begin{array}{l}
       E_1+ \ubx_{2,1}- \lbx_{1,2}= 2 \\
       \text{and} \        E_2+ \ubx_{1,1}- \lbx_{2,2}\geq  2
     \end{array}
     \text{ OR }
          \begin{array}{l}
       E_2+ \ubx_{1,1}- \lbx_{2,2}= 2 \\
       \text{and} \ E_1+ \ubx_{2,1}- \lbx_{1,2}\geq 2 \ .
     \end{array}
     \end{align}
     With these conditions, the parameters values $(E^*_1,\ubx^*_{1,1},\ubx^*_{1,2},\lbx^*_{1,1},\lbx^*_{1,2})$ are no longer possible for agent $1$ (neither for agent 2 as the information the operator obtains is necessarily symmetric).
     Indeed, from the initial aggregate conditions, this implies that $ (E^*_2,\ubx^*_{2,1},\ubx^*_{2,2},\lbx^*_{2,1},\lbx^*_{2,2}) \eqd (1,1,0,0,0)$, which  is incompatible with \eqref{eq:cond-op-privacy}.
   \end{remark}
   
   It seems difficult to have an algorithm where the operator will not learn any information on the distribution of parameters.
   The kind of information leak illustrated in \Cref{rm:info-leak}  is also  observed in other privacy-preserving algorithms. One such example is the privacy-preserving consensus method to compute an average value of numbers secretly owned by agents of  \cite{ruan2019secure}.
   Even if the secret number of a particular agent cannot be learned by another agent (see \cite[Th.3]{ruan2019secure}),  at the end of the procedure, each agent has obtained new information on the joint distribution of initial numbers of other agents, as each agent obtains the average value Avg[0] of those numbers.

Fortunately, in most cases, an operator using \Cref{algo:confidOptimDisag} will only learn very limited information on the distribution of  constraints and profiles of users: for this, 
it is instructive to estimate the information leak caused by the
algorithm by comparing the information sent to the central
operator with the one needed to encode the instance.
Suppose, for simplicity, that all the numbers manipulated
by the algorithm (in particular
the bounds $\underline{x}_{n,t}$ and $\bar{x}_{n,t}$
are encoded by fixed (say double) precision numbers.
Then, the total number of bits needed to encode
the instance is $O(NT)$. However, it follows from~\Cref{prop:correctnessAlgo}
that the number of bits received by a malicious central operator
is $O(2^T)$, hence, in the regime $N\gg 2^T$ (large number
of agents), the information
leak is asymptotically negligible. 
Moreover, we give in~\Cref{subsec:appli-microgrid}
a real example with $N=2^8$ agents, $T=24$,
in which the algorithm terminates after 194 iterations.
Here, the information leak consists of 194 doubles,
to be compared with the encoding size of the instance, $NT+N=12544$ doubles.

\paragraph{Semi-honest agents}
Let us now consider agents that
are only semi-honest, meaning that they still obey the protocol,
but wish to exploit the data received in the algorithm to infer information about the
other users.

A semi-honest agent (or an external eavesdropper) aiming to learn the profile $\xx_n$ of agent $n$ has to intercept all the communications between $n$ and all $N-1$ other agents (to learn $(s_{n,t,m})_{m\neq n}$ and $(s_{m,t,n})_{m\neq n}$) and to  the operator (to learn $\sigma_n$) during the SMC protocol in order to succeed (otherwise, this semi-honest agent has just access to random numbers).
Besides, if we think of collusion of semi-honest agents aiming to learn the profile of a specific agent, we can observe that the collusion must involve $N-1$ agents (all except one) to succeed in learning anything. 

Indeed, in the SMC summation protocol (\Cref{algo:SMCsum},
the messages
$s_{n,t,j}$ sent by agent $n$ to agent $j$, encoding 
shares of the consumption $x_{nt}$ of agent $n$ at different times, are all uniformly distributed
random variables, because the sum of independent uniform random variables
modulo $R$ is still uniform. Similarly, the messages $\sigma_{n,t}$
sent by agent $n$ to the central operator are uniformly distributed random variables.
Furthermore, there
are no correlations between the random variables
received by an agent at different iterations. This entails
that, even if some agents keep the same consumption
profile over time, a semi-honest agent cannot learn the profiles
of the other agents by comparing the information received
during different iterations of the SMC protocol.

We refer the reader to
\cite[Protocol I]{atallah2004private}, where a finer variant of SMC (secure split protocol), splitting secret numbers only among $k<n$ randomly chosen players, is analyzed. %

Alternative SMC methods to the proposed protocol \Cref{algo:SMCsum} exist 
to compute the sum $\sum_{n\in \Nall} y_n$,
while keeping $y_n$ secret to user $n$, for instance \cite{clifton2002tools}.
Other techniques could also be considered such as consensus-based aggregation algorithms \cite{heCai2019consensus}. %

From the above paragraph, one deduces that a collusion of less than $(N-1)$ semi-honest agents cannot obtain more information than the operator from an execution of \Cref{algo:confidOptimDisag}: indeed, the only additional information that the collusion obtains is composed of the sequences $\{\nnu^{(k)}\}$ obtained from each execution of \Cref{algo:vonNeumannProjasymConfid}.
Thus,  the symmetry argument given in \Cref{prop-safe} applies to give a privacy guarantee similar to \Cref{cor:safeOperator}:
\begin{corollary}\label{cor:safeAgents}
Using \Cref{algo:confidOptimDisag} with $N>1$, a collusion of less than $(N-1)$ malicious (or semi-honest) agents cannot infer the constraints $\X_n$ or profile $\xx_n$ of a particular agent $n $ (which is not in the collusion)  with probability 1.
  \end{corollary}

\paragraph{Malicious agents and Robustness}
The proposed procedure  is not robust against malicious agents,
i.e., agents lying about their consumption to jeopardize the system.
Indeed, if agents lie about their consumptions or feasibility constraints, \Cref{algo:confidOptimDisag} will not converge to the global optimum.
However, we still have the privacy guarantee given in \Cref{cor:safeOperator}:  malicious agents lying to the operator in order to learn information about other agents  can only learn \emph{global} information, and will not be able to learn \emph{individual} information (profile or constraints) about a specific agent.

A more detailed privacy analysis of our method and of its possible refinements
can be an avenue for further work.

In the next section, we focus on the convergence rate of APM in the particular case of transportation constraints,
and give an explicit bound on the geometric rate stated in \Cref{thm:cvgAP}.

\subsection{Complexity Analysis of APM in the Transportation Case}
\label{subsec:complexityAnalysis}
In this section we analyze the speed of convergence of the alternate projections method (APM) described in \Cref{algo:vonNeumannProj} on the sets $\X$ and $\Y_{\pp}$ defined in \Cref{sec:resourceallocation-disag}.

A general result in \cite{borwein2014analysis} gives an upper bound of the sequences generated by APM on 
$\X$ and $\Y$ if these two sets are semi-algebraic. 
In particular,  it establishes the geometric convergence for polyhedral sets. 
However, as stated in \cite{nishihara2014convergence}, given two particular polyhedral sets $\X$ and $\Y$, it is not straightforward to deduce an explicit rate of  convergence from their result.

The authors in \cite{nishihara2014convergence} established in a particular case a geometric convergence with an explicit upper bound on the convergence rate. 
They consider APM on two sets $P$ and $Q$, where $P$ is a linear subspace and $Q$ is a product of base polytopes of submodular functions.

In this section, we also establish an explicit upper bound on the convergence rate of APM in the transportation case, that is with $\X$ and $\Y_{\pp}$ defined in \eqref{eq:XnTransport} and \eqref{eq:defY}:
\begin{theorem} \label{prop:boundCvgRate} 
For the two sets $\X$ and $\Y_{\pp}$, the sequence of alternate projections converges  to $\xE^*\in \X$, $\xP^*\in \X^P$ satisfying $\norm{\xE^*-\xP^*}=\inf_{\xx\in\X,\yy\in\Y_{\pp}} \norm{\xx-\yy}$, at the geometric rate:
\begin{equation*}
\norm{ \xE^{(k)}-\xE^*} \leq 2 \norm{ \xE^{(0)}-\xE^*} \times \left(1- \tfrac{4}{N(T+1)^2(T-1)} \right)^k ,
\end{equation*}
and the analogous inequalities hold for $(\xP\kexp)_k$.
\end{theorem}
 For the remaining of this section, we will just use $\Y$ to denote $\Y_{\pp}$, as $\pp$ remains fixed during   APM. 
 For the result stated in \Cref{prop:boundCvgRate} above, we use several
 lemmas from \cite{nishihara2014convergence}. 

\begin{proof}
First, we use the fact stated in \cite{nishihara2014convergence} that APM on subspaces $U$ and $V$ converge with geometric rate $c_F(U,V)^2$, where the rate is given by the  square of the cosine of the Friedrichs angle  between $U$ and $V$, given by:
\begin{equation*}
c_F(U,V)= \sup \{u^T v \ | \ u \in U \cap (U \cap V)^\perp, v \in V \cap (U \cap V)^\perp, \norm{u} \leq 1, \norm{v} \leq 1 \}.
\end{equation*}
An intuitive generalization of this result for polyhedra $\X$ and $\Y$, considering all affine subspaces supporting the faces of $\X$ and $\Y$ is given in \cite{nishihara2014convergence}:
\begin{lemma}[\cite{nishihara2014convergence}]
\label{lem:nishiharaFriedrich}
For APM on polyhedra $\X$ and $\Y$ in $\rr^D$, the convergence is geometric with rate bounded by the square of the maximal  cosine of Friedrichs angle between subspaces supporting faces of $\X$ and $\Y$:
\begin{equation} \label{eq:maxFangle}
\max_{\xx,\yy} c_F\big( \aff_0(\X_{\xx}), \aff_0(\Y_{\yy}) \big),
\end{equation}
where, for any $\xx \in \rr^D$,  $\X_{\xx} \eqd \arg \max_{\bm{v} \in \X} \xx^\tr  \bm{v}$ is the face of $\X$ generated by direction $\xx$ and $\aff_0(C)=\aff(C)-\cc$ for some $\cc \in C$  denotes the  subspace supporting the affine hull of $C$, for $C=\X_{\xx}$ or $C=\Y_{\yy}$.
\end{lemma}
In the remaining of the proof, we bound the quantity \eqref{eq:maxFangle} for our polyhedra $\X$ and $\Y$.

For this, we use the space $\rr^{NT}= \rr^T \times \dots \times \rr^T$, where the $(\n-1)T+1$ to $\n T$ entries correspond to the profile of agent $\n$, for $1\leq \n \leq N$. 
As in \cite{nishihara2014convergence}, we use a result connecting angles between subspaces and the eigenvalues of matrices giving the directions of these spaces:
\begin{lemma}[\cite{nishihara2014convergence}]
\label{lem:nishiharaSpectral}
If $A$ and $B$ are matrices with orthonormal rows with same number of columns, then:
\begin{itemize}
\item[] -- if all the singular values of $AB^\tr$ are equal to one, then $c_F\big( \Ker A, \ \Ker B \big)=0$;
\item[] -- otherwise,  $c_F\big( \Ker A, \ \Ker B \big)$ is equal to the largest singular value of $AB^\tr$ among those that are smaller than one.
\end{itemize} 
\end{lemma}
We are left with finding a matrix representation of the faces of polyhedra $\X$ and $\Y$ and, then, bounding the corresponding singular values.

In our case, the polyhedra $\Y$ is an affine subspace $\Y=\{ \xx \in \rr^{NT} \ | A\xx= \sqrt{N}^{-1} \pp \}$ where:
\setcounter{MaxMatrixCols}{20}
\begin{equation*}
A\eqDef \sqrt{N}^{-1} J_{1,N} \otimes I_T , 
\end{equation*}
where $\otimes$ denotes the Kronecker product. The matrix $A$ has orthonormal rows and the linear subspace associated to $\Y$ is equal to $\Ker (A)$.

Obtaining a matrix representation of the faces of $\X$ is more complex. 
The faces of $\X$ are obtained by considering, for each $\n\in \Nall$,  subsets of the time periods that are at lower or upper bound (respectively $\Tun $ and $\Ton $, with $\Tun  \cap \Ton = \emptyset $).  Considering a collection of such subsets, a face of $\X$ can be  written as:
\begin{equation*}
\mathcal{A}_{(\Ton,\Tun)_n}  \hh\hh\eqDef \hh \Big\{(\xx)\nt \ | \forall n, \ \txt\sum_t \x\nt = E_n \text{ and } \forall t \in \Tun, \x\nt =\lbx\nt,     \text{ and } \forall t \in \Ton, \x\nt =\ubx\nt \Big\}\ .
\end{equation*}
For some particular collection of subsets $(\Ton,\Tun)_n$, the set $\mathcal{A}_{(\Ton,\Tun)_n}$ might be empty. 
The linear subspace associated to $\mathcal{A}_{(\Ton,\Tun)_n}$ is given by $\{\xx \in \rr^{NT} | \ B\xx=0 \} = \Ker(B) $, where the $N$ first rows of $B$, corresponding to the constraints  $\sum_t \x\nt = E_n$, are given before orthonormalization by:
\begin{equation*}
\sqrt{T}^{-1} I_N \otimes J_{1,T} \ ,
\end{equation*}
and the matrix $B$ has $b\eqDef \sum_n |\T_n| $ more rows, where $\T_n \eqd \underline{\T_n} \cup \ol{\T_n}$, corresponding to the saturated bounds. Each of this row is given by the unit vector $\ee\nt^\tr \in \rr^{NT}$  for $\n\in \Nall$, $t \in  \T_n$, which gives already an orthonormalized family of (unit) vectors. Therefore, a simple orthonormalized matrix  $B\in \mc{M}_{ N+b, NT}(\rr)$ giving the direction of $\mathcal{A}_{(\Ton,\Tun)_n}$ is given by:
\begin{equation*} 
B \eqd 
\begin{pmatrix}
 \diag\left( \sqrt{T-|\T_1|}^{-1} {\oneT}_{\T_1^c}^\tr , \dots ,  \sqrt{T-|\T_N|}^{-1}{\oneT}_{\T_N^c}^\tr \right)   \ & | & \ 
\diag\left(  B_{\T_1} , \dots , B_{\T_N}  \right)
\end{pmatrix}^\tr
\ ,\end{equation*} 
where $\oneT_{\T_n^c} \in \rr^T$ is the vector  where the indices in $\T_n^c$ are equal to 1 and 0 otherwise, and $B_{\T_n} \eqd \sum_{\substack{ 1 \leq k \leq |\T_n| \\ \T_n= \{t_1,\dots, t_{|\T_n|} \} } } E_{k t_k}$ is the matrix $|\T_n| \times T$ with indices of $\T_n$. 
We obtain 
 the double product:
\begin{align*}
(A B^\tr) (B A^\tr) & = \frac{1}{N} \begin{pmatrix}
\sum_n \dfrac{\dsone_{ k \notin \T_n \wedge \ell \notin \T_n }}{ T-|\T_n| }
\end{pmatrix}_{ 1\leq k,\ell\leq T} + \frac{1}{N}\sum_n B_{\T_n}^\tr B_{\T_n} \\
& =\frac{1}{N} \begin{pmatrix}
\sum_n \dfrac{\dsone_{ \{k,\ell\} \subset \T_n^c }}{ T-|\T_n| }
\end{pmatrix}_{ 1\leq k,\ell\leq T} + \frac{1}{N} \sum_{1 \leq t \leq T} \left(\txt\sum_n \dsone_{t\in \T_n}\right) E_{t,t} \ .
\end{align*}
We observe that:
\begin{itemize}[wide,labelindent=0pt]
\item[] -- if $t_0 \in \bigcap_{i=1}^N \T_n $, then $\ee_{t_0}$ is an eigenvector associated to eigenvalue $\lambda_{t_0}=1$ ;
\item[] -- the vector $\oneT_{\barT}\eqd (\dsone_{t\notin \cap_n \T_n})_{t\in \Tall} \in \rr^T$, where $\barT \eqd \cup_n \T_n^c$, is an eigenvector associated to eigenvalue $\lambda=1$. Indeed, if we denote by $\N_\theta= \{n  \in \Nall| \theta \in  \T_n \}$, then $[\oneT_{\barT}]_\theta=1 \Leftrightarrow \N_\theta^c \neq \emptyset$, and for each $\theta\in\barT$:
\begin{align*}
[(A B^\tr) & (B A^\tr)]_\theta \oneT_{\barT}= \frac{1}{N} \left( \sum_{i\in \N_\theta^c} \sum_t \frac{\dsone_{t\notin \T_n}}{T-|\T_n|} [\oneT_{\barT}]_t +\sum_n \dsone_{\theta \in \T_n} [\oneT_{\barT}]_{\theta} \right) \\
& = \frac{1}{N} \left( \sum_{i\in \N_\theta^c} \frac{T-|\T_n|}{T-|\T_n|} 1 +\sum_{i\in \N_\theta} 1 \times [\oneT_{\barT}]_\theta\right) = \frac{|\N_\theta^c| + |\N_\theta| [\oneT_{\barT}]_\theta}{N} = [\oneT_{\barT}]_\theta \ .
\end{align*}
\end{itemize}

To bound the other eigenvalues of the matrix $(A B^\tr) (B A^\tr)$, we rely on spectral graph theory arguments. Consider the weighted graph $\G=(\Tall,\E)$ whose vertices are the time periods $\Tall$ and each edge $(k,\ell) \in \Tall \times \Tall$ with $k \neq \ell$ has weight $S_{k,\ell}= \tfrac{1}{N}\sum_n \tfrac{\dsone_{ \{k,\ell\} \subset \T_n^c }}{ T-|\T_n| }$ (if this quantity is zero, then there is no edge between $k$ and $\ell$). 

The matrix $P \eqd I_T-(A B^\tr)  (B A^\tr)$ verifies for each $k \in \Tall$:
\begin{align}
& \sum_{\ell \neq k }  -P_{k,\ell} = \sum_{\ell \neq k } \tfrac{1}{N}\sum_n \dfrac{\dsone_{ \{k,\ell\} \subset \T_n^c }}{ T-|\T_n| } 
  = \tfrac{1}{N}\sum_n \dfrac{\dsone_{ k \in \T_n^c } (T-|\T_n| -1) }{ T-|\T_n| }  \\
  &     = \tfrac{1}{N}\sum_n (1-\dsone_{ k \in \T_n }) -   \tfrac{1}{N}\sum_n \dfrac{\dsone_{ k \in \T_n^c } }{ T-|\T_n| } = P_{kk}   \ , 
\end{align}
which shows that $P$ is the Laplacian matrix of graph $\G$. 
As Sp$(A B^\tr  B A^\tr)=1-$Sp$(P)$, we want to have a lower bound on the least eigenvalue of $P$ greater than 0, that we denote by $\lambda_1$. 

By rearranging the indices of $\Tall$ in two blocs $\barT$ and $\barT^c$, we observe that $P$ can be written as a block diagonal matrix $P=  \text{diag}(P_{\barT}, 0_{\barT^c} )$. As we are only interested in the positive eigenvalues of $\P$, we can therefore study the linear application associated to $P$ restricted to the subspace $\operatorname{Vect}(e_t)_{t\in\barT}$.

As $\oneH$ is an eigenvector of $P$ associated to $\lambda_0=0$, from the minmax theorem, we have:
\begin{equation} \label{eq:lambda1-minmax}
\lambda_1 = \min_{u \perp \oneH, u \neq 0 } \dfrac{ u^\tr P u}{ u ^\tr u} \ .
\end{equation}
Let us consider an eigenvector $u$ realizing  \eqref{eq:lambda1-minmax}. Let $u_{t^*}\eqd \max_t u_t$ and $u_{s^*}\eqd \min_t u_t$  and let $d_{s^*,t^*}$ be the distance between $s^*$ and $t^*$ in $\G$, and let $(s^*\tm t^*)$ denote a shortest path from $s^*$ to $t^*$ in $\G$.
 As $P$ is a Laplacian matrix, we have:
\begin{align}
& u^\tr P u  = \frac{1}{2}\sum_{k , \ell \in \barT}  \hh\hh-P_{k,\ell} (u_k - u_\ell)^2  \geq \frac{1}{2}\hh\hh \sum_{ \{k , \ell\} \in (s^*\text{-}t^*)} \hh\hh\hh\hh\hh\hh\hh\hh -P_{k,\ell} (u_k - u_\ell)^2 %
 \geq   \min_{k,\ell \in (s^*\text{-}t^*)} (-P_{k,\ell})  \dfrac{(u_{t^*}- u_{s^*})^2 }{ d_{s^*,t^*}} \ , \label{eq:bounduPu}
\end{align}
where the last inequality is obtained from Cauchy-Schwartz inequality. 

Let us write the path $(s^*\text{-}t^*)=(t_0,t_1,\dots,t_d)$. As $(s^*\text{-}t^*)$ is a shortest path, for each $k\in \{0,d-1\}$, the edge $(t_k,t_{k+1})$ exists so there exists $n\in\Nall$ such that $ \{t_k,t_{k+1}  \} \subset \T_n^c$. Moreover, for each $n$, we have $\T_n^c \cap \{t_0, \dots, t_{k-1}, t_{k+2}, \dots , t_d \} = \emptyset$, otherwise we could ``shortcut'' the path $(s^*\text{-}t^*)$, thus we have $|\T_n| \geq d-1$.
We obtain:
\begin{equation*}
- P_{t_k,t_{k+1}} = \frac{1}{N}\sum_n \dfrac{\dsone_{ \{t_k,t_{k+1}\} \subset \T_n^c }}{ T-|\T_n| }  \geq \frac{1}{N (T-d+1)} \ .
\end{equation*} 
 On the other hand, we have $(u_{t^*}- u_{s^*}) \geq u_{t^*} + \tfrac{u_{t^*}}{T-1}= \tfrac{T}{T-1} u_{t^*} \geq \tfrac{T}{(T-1)\sqrt{T}}\norm{u}_2$.  

 Using these bounds and  \eqref{eq:bounduPu}, we obtain:  
\begin{align*}
& u^\tr P u \geq \frac{(u_{t^*}- u_{s^*})^2}{N(T-d_{s^*,t^*}+1) d_{s^*,t^*}}   \geq \frac{4T}{N(T+1)^2(T-1)^2} \norm{u}_2^2  \geq \frac{4}{N(T+1)^2(T-1)} \norm{u}_2^2 .
\end{align*} 
Therefore, $\lambda_1\geq \frac{4}{N(T+1)^2(T-1)} \eqd \kappa_{N,T}$ and the greatest singular value lower than one of $(A B^\tr) (B A^\tr)$ is $1-\kappa_{N,T}$.
We conclude by applying successively \Cref{lem:nishiharaSpectral} and \Cref{lem:nishiharaFriedrich}, to obtain the  convergence rate stated in \Cref{prop:boundCvgRate}.
\end{proof}
\begin{figure}[ht!]
\centering
\vspace{-0.5cm}
\includegraphics[width=0.5\textwidth]{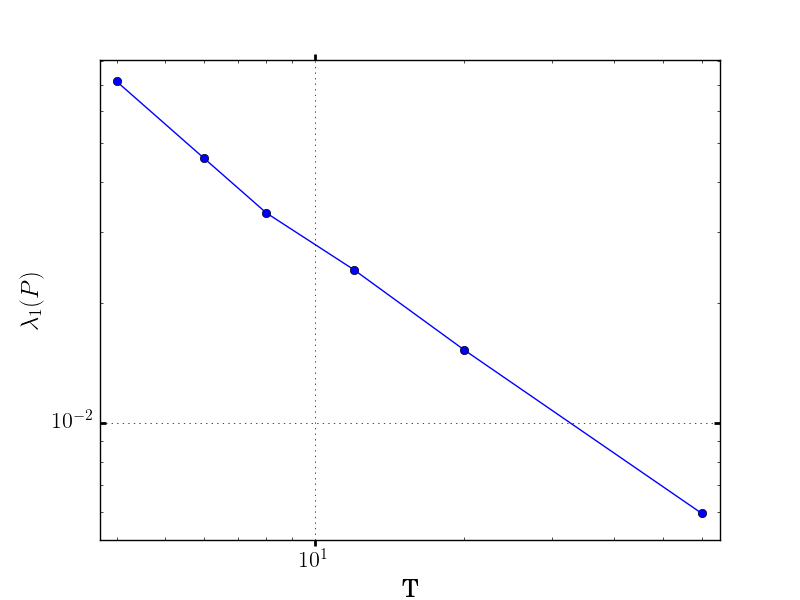}
\caption{Evolution of the convergence rate, given as $\lambda_1(P)$ (lowest nonzero eigenvalue of $P$), with $N=6$ and $T \in \{4,6,8,12,20,60\}$. \textit{The worst convergence rate is evaluated by taking $100 \times T$  random draws of the sets $\T_n \subset \Tall$ for each $n$, and evaluating the eigenvalue of the matrix. The slope is around -0.93, which indicates that in practice the convergence rate is $\mathcal{O}(T^{-1})$, faster than the upper bound in $\mathcal{O}(T^{-3})$ established in  \Cref{prop:boundCvgRate}.} \vspace{-0.5cm}}\end{figure}

\section{Generalization to Polyhedral Agents Constraints} \label{sec:generArbitraryIndividualSet}

In this section, we extend our results to a more general framework where for each $n\in\Nall$, $\X_n$ is an arbitrary polyhedron, instead of having the particular structure given in \eqref{eq:XnTransport}. 
Let us now consider that $(\X_\n)_\n$ are  polyhedra with, for each $n$:
\begin{equation} \label{eq:Xnpolyhedral}
\X_\n= \{ \xx_\n \in \rr^T | A_\n \xx_\n \leq \bb_\n \} \ ,
\end{equation}
with $A_\n  \in \mc{M}_{k_n,T}(\rr)$ with $k_n \in \nit$. The disaggregation problem \eqref{pb:disag}, \emph{with} $\pp \in \P$ \emph{fixed}, writes:
\begin{subequations} \label{pb:disag-general}
\begin{align}
& \min_{\xx \in \rr^{NT}} 0 \\
\text{s.t.} \ &  A_0 \xx= B \pp  \quad (\llam_0) \label{eq:c-ag2}\\ 
 &  A_\n \xx_\n \leq \bb_\n , \ \forall \n \in \Nall\quad (\llam_\n) \label{eq:c-cons2}\ .
\end{align}
\end{subequations}
where $A_0=J_{1,N} \otimes I_T$, $B=I_T$,   (such that \eqref{eq:c-ag2} corresponds to the aggregation constraint $\sum_n \xx_n =\pp$) and  $\llam_0 \in \rr^T$, $(\llam_\n)_{\n\in\Nall} \in \rr_+^{\sum_n k_n}$ are the Lagrangian multipliers associated to \eqref{eq:c-ag2} and \eqref{eq:c-cons2}. %

 With the polyhedral constraints  \eqref{eq:Xnpolyhedral}, the graph representation of the disaggregation problem, as illustrated in \Cref{fig:graphDisagregationEx} is no longer valid. 
Consequently, one can not directly apply Hoffman's theorem (\Cref{th:hoffman}) to obtain a characterization of disaggregation feasibility by inequalities on $\pp$. 
However, using duality theory,   \Cref{prop:genLambdaChar} below also gives  a characterization of disaggregation:
\begin{proposition} \label{prop:genLambdaChar}
A profile $\pp \in \P$ is disaggregable iff:
\begin{equation}
\forall \big(\llam_0, \llam_1, \dots \llam_N \big) \in \Lambda, \ \ \llam_0^\tr \Bp \pp  + \txt\sum_{\n\in\Nall} \llam_\n^\tr  \bb_\n \geq 0 \ , \label{eq:condLambda}
\end{equation}
where
\[ 
\Lambda \eqd \{ \llam_0 \in \rr^{k_0}, \forall n\in\Nall, \llam_\n \in \rr^{k_\n}_+  \ | \   A_0^\tr \llam_0 +  A^\tr (\llam_n)_n = 0 \},
\]
with $A \eqDef \diag(A_\n)_{\n\in\Nall}$.
\end{proposition}
\begin{proof}
From strong duality, we have:
\begin{align}
& \min_{\xx \in \rr^{NT}}  \max_{ \llam_0 \in \rr^{k_0}, \llam_\n \in \rr^{k_\n}_+} \llam_0^\tr (A_0\xx - \Bp \pp)  + \sum_\n \llam_\n^\tr (A_\n \xx_\n  - \bb_\n ) \\ 
& = \begin{array}{l}
\displaystyle\max_{ \llam_0 \in \rr^{k_0}, \llam_\n \in \rr^{k_\n}_+} - \llam_0^\tr \Bp \pp  - \sum_\n \llam_\n^\tr  \bb_\n  \\
\text{s.t. }  \llam_0^\tr A_0 + (\llam_n)_n^\tr A = 0
\end{array}   \label{eq:pbgendual} \ .
\end{align}
If the  polytope $\Y_{\pp} \cap \X$ given by the constraints of \eqref{pb:disag-general} is empty, then there is an infeasibility  certificate $ \llam^\tr= (\llam_0^\tr\ \llam_1^\tr \dots \llam_N^\tr) \in \rr^{T} \times \prod_n \rr_+^{k_\n}$ such that:
\begin{align}
\label{eq:cond-inf}
& \llam_0^\tr A_0 + (\llam_n^\tr)_n A = 0 %
\text{ and } 
 \llam_0^\top \Bp  \pp + (\llam_n)_n^\top \bb <0  \ . %
\end{align}
On the other hand, if $\Y_{\pp} \cap \X$ is nonempty, then a solution to the dual problem  \eqref{eq:pbgendual} is bounded, which implies that $\forall \llam\eqd (\llam_0,(\llam_n)_n) \in \Lambda, \ \ \llam_0^\tr \Bp \pp  + \sum_\n \llam_\n^\tr  \bb_\n \geq 0 \ $.

\end{proof}
As opposed to Hoffman circulation's theorem where disaggregation is characterized by a finite number of inequalities, \Cref{prop:genLambdaChar} involves a priori an infinite  number of inequalities. 

However, we know that the polyhedral cone $\Lambda$ can be represented by a finite number of generators (edges), that is, there exists $\Lambda^* \eqd \{ \llam^{*(1)}, \dots, \llam^{*(d)} \}$  such that:
\begin{equation}
\Lambda = \Big\{ \txt\sum_{1 \leq i \leq d} \alpha_i \llam^{*(i)} \ | \ (\alpha_i)_i \in \rr_+^d \Big\} \ .
\end{equation} 
Thus, we obtain the following corollary to \Cref{prop:genLambdaChar}:
\begin{corollary}\label{cor:condLambdaFinite} There exists a finite set $\Lambda^* \subset \Lambda$ such that, for any $\pp \in \P$, $\pp$ is disaggregable iff:
\begin{equation}\label{eq:condLambdaFinite}
\forall (\llam_0, (\llam_n)_n) \in \Lambda^*, \ \ \llam_0 \Bp \pp  + \txt\sum_\n \llam_\n  \bb_\n \geq 0 \ . 
\end{equation}

\end{corollary} 

\begin{remark}%
In the transportation case \eqref{eq:XnTransport}, we can write each agent constraints in the form $A_\n \xx_\n \leq \bb_\n$  (writing the equality $\sum_t x\nt = E_n$ is written as two inequalities), and Hoffman conditions \eqref{eq:hoffmanCondReverse} can be written in the form \eqref{eq:condLambdaFinite}. 
Moreover, \Cref{th:cutN0T0violated} ensures that one possibility for $\Lambda^*$ of \Cref{cor:condLambdaFinite} is to consider the collection of $2^T$ multipliers corresponding to the subsets $\Tcut \subset \Tall$ and $\Ncut$ minimizing \eqref{eq:ineq2powTv1}. We skip the details here for brevity. 
\end{remark}

As in the first part of the paper, we want to use APM to decompose problem \eqref{pb:global} and, in the case where disaggregation is not possible,  use the result of APM to generate 
an  inequality \eqref{eq:condLambda} violated by the current profile $\pp$.

In the case of impossible disaggregation, APM converges to the orbit $(\yy^\infty,\xx^\infty)$, and $\mmu \eqd \yy^\infty- \xx^\infty$ defines a separating hyperplane $ \bxx +  \mmu^{\perp}$ , where $\bxx= \frac{ \yy^\infty  + \xx^\infty}{2}$, that satisfies, with  $a \eqDef  \bxx.\mmu $  (note that $\bxx$ can be replaced by any  $\yy \in [ \yy^\infty, \xx^\infty]$):
\begin{align*}
& \forall \xx \in \Y_{\pp}; \  \mmu^\tr \xx  > a   
& \forall \xx \in \X ; \  a > \mmu^\tr \xx  ,
\end{align*}
which give lower bounds on the linear problems (the second one is decomposed because $A$ is a block-diagonal matrix, but it can also be written in one  problem):
\begin{align} 
\label{eq:pbToFindLambdas}
& \begin{array}{l}
\displaystyle\min_{\xx\in\rr^{NT} }\mmu^\tr \xx  \\
A_0 \xx = \Bp \pp \ (\llam_0)
\end{array}  \ = 
\begin{array}{l}
\displaystyle\max_{\llam_0 \in \rr^{k_0}} - \llam_0^\top \Bp  \pp \\
\mmu= -A_0^\top \llam_0 
\end{array} \\
 \label{eq:pbfind-lambdaUsers}
\text{and } \quad  \forall \n \in \Nall, \ \ &  \begin{array}{l}
\displaystyle\max_{\xx\in\rr^{NT} }\mmu_\n \xx_\n  \\
A_\n \xx_\n \leq \bb_\n \ (\llam_\n)
\end{array} \ =
\begin{array}{l}
\min_{\llam_\n \in \rr_+^{k_\n}} \bb_\n^\top \llam_\n \\
 \mmu_\n= \llam_\n^\top A_\n 
\end{array} \ .
\end{align}
Strong duality on these problems implies that there exist $\llam_0 $ and $\llam$ such that:
\begin{align}
& \mmu= -A_0^\top \llam_0  \text{ and }  a < - \llam_0^\top \Bp  \pp   %
& \mmu= \llam^\top A \text{ and } a >  \bb^\top \llam \label{eq:muFarkexistLam} \ .
\end{align}
Thus, we obtain $(\llam_0,\llam) \in \Lambda$   satisfying \eqref{eq:cond-inf}, that is, $\llam_0^\top \Bp  \pp  + \bb^\top \llam <0$,  and we can use this to add a new valid additional inequality on $\pp$ of form \eqref{eq:condLambda}, that will change the current profile $\pp$:
\begin{equation}
\llam_0^\top \Bp  \pp + \llam^\top \bb  \geq 0 
\end{equation}

In \Cref{algo:decompAPM-polyhedral} below, we  summarize the proposed decomposition of problem \eqref{pb:global}. This is a generalization of the decomposition principle used for  \Cref{algo:confidOptimDisag}.
\begin{algorithm}[!ht]
\begin{algorithmic}[1]
\Require Start with $\CutSet^{(0)}=\{\}$, $k=0$, \dsg= $\False$ \;
\While{not \dsg} 
\State get solution $\pp^{(k)}$ of problem $\min_{\pp \in \P} \{f(\pp)\  | \  \llam_0^\top \Bp  \pp + \llam^\top \bb  \geq 0  , \ \forall \llam \in \CutSet^{(k)}  \} $  \;
\State get $\mmu^{(k)}= \yy^\infty-\xx^\infty  \leftarrow $ APM($\Y_{\pp^{(k)}},\X)$ \;
\If{$\mmu^{(k)}\neq 0$} \;
\State obtain $\llam_0^{(k)} \leftarrow  \max_{\llam_0 \in \rr^{k_0}} \{ -  \llam_0^\top \Bp  \pp^{(k)} \ | \ \mmu^{(k)}= -A_0^\top \llam_0  \}$ \;
\State obtain for each $\n$, $\llam^{(k)}_\n \leftarrow \max_{\llam_\n \geq 0} \{ \bb_\n^\top \llam_\n  \ | \  \mmu_\n^{(k)}= \llam_\n^\top A_\n  \}$\;
\State add  $ \CutSet^{(k+1)} =\CutSet^{(k)} \cup \{(\llam_0^{(k)}, \llam^{(k)})\} $ \;
\Else
\State Return $\dsg=\True$, $\pp^{(k)}$ as optimal solution
\EndIf
\State $k \leftarrow k+1$ \;
\EndWhile \;
\end{algorithmic}
\caption{Non-intrusive optimal disaggregation with polyhedral constraints }
\label{algo:decompAPM-polyhedral}
\end{algorithm}

\begin{remark} \label{rem:precisionFiniteAlgoGen}
We use the fact that $\mmu=\yy^\infty-\xx^\infty$ although, as before, we only have an \emph{approximation } of this quantity. The approximation has to be precise enough to ensure that the solution obtained verifies $\llam_0^\top \Bp  \pp  + \bb^\top \llam <0$. In practice, one can proceed as in the transportation case and \Cref{algo:vonNeumannProjasymConfid} use  a large $\epscvg$ as stopping criteria in APM, then compute $(\llam_0,\llam) \in \Lambda$ and check if $\llam_0^\top \Bp  \pp  + \bb^\top \llam <0$. If this is not the case, restart with $\epscvg= \epscvg/2$.
\end{remark}

\begin{remark}
When $\Y_{\pp}=\{ \xx \in \rr^{NT} | \Arhs_0 \xx = \Bseuil_0 \pp \} = \{ \xx | \sum_{n} \xx_n = \pp \} $, we can obtain a non-intrusive version of APM on $\Y_{\pp}$ and $\X$, similar to \Cref{algo:vonNeumannProjasymConfid}. 
 In this case, \eqref{eq:muFarkexistLam} ensures that we have  $\mmu_{\n,t}=-[{\llam_0}]_t$ for each $\n\in\Nall$, and $\llam_0$ is fixed by $\mmu$.
The only difference with the transportation case  for a non-intrusive APM in the general polyhedral case, is in the way of computing the valid constraint violated by $\pp$. Thus, %
Lines 16 to 19 of \Cref{algo:vonNeumannProjasymConfid} have to be replaced by \Cref{algo:vonNeumannProjasymConfid-poly}.
\begin{algorithm}[!ht]
\begin{algorithmic}[1]
\makeatletter
\setcounter{ALG@line}{15}
\makeatother
\For{each agent $n \in \Nall$} 
\State compute $M_n$ optimal value of \eqref{eq:pbfind-lambdaUsers}.
\EndFor
\State Operator computes $M \leftarrow $ SMCA$( (M_n)_n)$   \label{algline:NIAPMpoly-compMagg}\;
\State \textbf{if} {$-\nnu.\pp  + M <0$} \textbf{ then} \label{algline:NIAPMpoly:checkconstraint}
\State \quad \Return \textsc{Disag} $\leftarrow$ \textsc{False}, $-\nnu, M$\; 
\end{algorithmic}
\caption{Modification of Lines 16-19 of \Cref{algo:vonNeumannProjasymConfid} for NI-APM with polyhedral constraints}
\label{algo:vonNeumannProjasymConfid-poly}
\end{algorithm}

\end{remark}

\begin{remark}[Link with Benders' Decomposition]
\label{subsec:benders}
In this generalized case, we obtain an algorithm related to Benders' decomposition \cite{benders1962partitioning} 
 (recall that in our specific case \eqref{pb:disag-general}, the cost function does not involve the variable $\xx$ but only variable $\pp$).

The difference between the proposed  \Cref{algo:decompAPM-polyhedral} and Benders' decomposition %
 lies in the way of generating the new cut. Benders' decomposition would directly solve the dual problem \eqref{eq:pbgendual}:  $  \max_{\llam} \{ - \llam_0^\tr B_0\pp  - (\llam_n)_n^\tr \bb  \ | \ \llam_0 A_0 + (\llam_n)_n A = 0 \}$ and obtain a cut if it is unbounded. 
 However, this problem involves the constraints of all users (through $A$ and $\bb$), and is it not straightforward to obtain a method to solve this problem in a decentralized and efficient way. 

\end{remark}
\section{Numerical Examples}
\label{sec:examples}
\subsection{An Illustrative Example with T=4}

\label{subsec:illusToyExample}

In this section we illustrate the iterations of the method proposed in this paper on an example with $T=4$ and $N=3$. 
Assuming that we have to satisfy the aggregate constraint $\sum_t p_t = \sum_n E_n$, we can use the projections on this affine space of solutions of master problems $(\pp^{(s)})_s$ to visualize them in dimension 3.

One can wonder if, in the transportation case, applying \Cref{algo:confidOptimDisag} or \Cref{algo:decompAPM-polyhedral} will always lead to the same cuts and solutions: the answer is no, as  shown by the instance considered in this section, for which \Cref{algo:confidOptimDisag} converges in 3 iterations and \Cref{algo:decompAPM-polyhedral} needs 4 iterations.

We consider the problem \eqref{pb:global} with agents constraints \eqref{eq:XnTransport} with parameters $\llbx\eqd 0$ and:
\begin{equation}
\uubx \eqd  \begin{matrix} [ 0.8, 0.2, 0.7, 0.1] \\
 [ 0.5, 0.1, 0.3, 0.6] \\
 [ 0.1, 0.1, 0.7, 0.2] \end{matrix}  \ \ , \ \  \begin{matrix}
E_1= 1.8 \\ E_2= 0.4 \\ E_3=1.1
\end{matrix}  \ , \quad \  \forall \pp \in \rr^4,  f(\pp )\eqd  \sum_{1 \leq t \leq 4} 0.8 \times p_t + 0.1 \times p_t^2 \   .
\end{equation}
Considering the aggregate equality constraint  $\sum_{ 1 \leq t  \leq 4 } p_t = \sum_{1 \leq n \leq 3} E_n = 3.3$, we use the canonical projection of  $4$ dimensional vectors into the 3 dimensional space $(p_1,p_2,p_3)$ to visualize the cuts and solutions.
In this example, there exist 
 $2^T-2=14$  nontrivial Hoffman inequalities characterizing disaggregation from \Cref{th:hoffman}. 
The projection of the obtained  polytope $\Popt$, as defined in \eqref{eq:def-Popt}, is represented in \Cref{fig:example-Popt}. One can remark that this polytope has only 6 facets.
Our \Cref{algo:confidOptimDisag} applied on this instance with $\epsdis=10^{-3}$ and $\epscvg=10^{-5}$ converges in 3 iterations, with successive solutions of the master problem \eqref{pb:master} and cuts added:
\begin{align*}
& \pp^{(1)}=[ 1.  ,  0.4 ,  1.  ,  0.9] & \overset{\text{cut}}{\longrightarrow}	\quad  p_1  + p_2 + p_4  \leq 1.9  \\
& \pp^{(2)}= [ 0.75 ,  0.4  ,  1.4  ,  0.75]   	 & \overset{\text{cut}}{\longrightarrow}	\quad  p_2 + p_3 +p_4 \leq 2.4  \\
&\pp^{(3)} = [ 0.9 ,  0.4 ,  1.4 ,  0.6]  \ .
\end{align*}
\Cref{fig:example-hoffman} represents in the projection space the three successive solutions and the two generated cuts (in red for each iteration).

On the other hand, applying \Cref{algo:decompAPM-polyhedral} with the same precision parameters$(\epsdis,\epscvg)$,  there are 3 cuts generated and 4 resolutions of master problem needed for convergence, given by (we refer the reader to \Cref{rem:precisionFiniteAlgoGen} for the numerical precision obtained in the values):
\begin{align*}
& \hat{\pp}^{(1)}=[ 1.  ,  0.4 ,  1.  ,  0.9] & 	\overset{\text{cut}}{\longrightarrow} \quad    -0.25 p_1  -0.25 p_2 + 1.0 p_3  -0.5 p_4  \geq 0.75 \\
& \hat{\pp}^{(2)}= [ 0.8097 ,  0.4    ,  1.3984 ,  0.6919]  & \overset{\text{cut}}{\longrightarrow} \quad  	 1.0 p_1  -0.509 p_2 + 0.018 p_3  -0.509 p_4 \geq 0.4161  \\
&\hat{\pp}^{(3)} = [ 0.9062 ,  0.4    ,  1.3823 ,  0.6115]  & \overset{\text{cut}}{\longrightarrow} \quad -0.333 p_1  -0.333 p_2 + 1.0 p_3 -0.333 p_4 \geq  0.7666 \\
& \hat{\pp}^{(4)}=[ 0.9 ,  0.4 ,  1.4 ,  0.6] \ .
\end{align*}
The 4 successive solutions and the 3 added cuts are represented in the three dimensional space on \Cref{fig:example-genAlgo}: we observe that the last cut needed to obtain the convergence of \Cref{algo:decompAPM-polyhedral}, corresponds to the first one added with \Cref{algo:confidOptimDisag}.

 Due to the strict convexity of the cost function $\pp \mapsto f(\pp)$, the final solution obtained is the same, unique aggregated optimal solution of \eqref{pb:global}. 
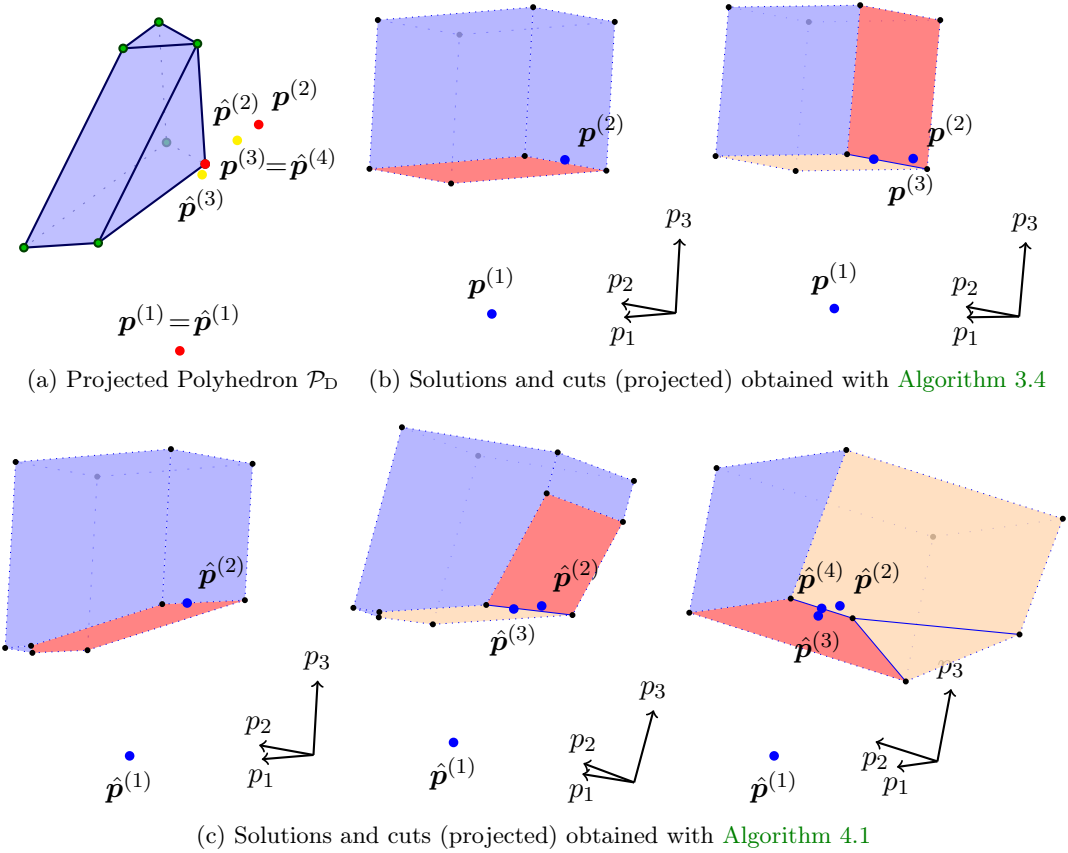
\begin{figure*}[ht!]
    \centering
    \subfloat[Projected Polyhedron $\Popt$]{   \label{fig:example-Popt}
\begin{tikzpicture}%
	[x={(-0.677335cm, -0.499961cm)},
	y={(0.734046cm, -0.410503cm)},
	z={(-0.048929cm, 0.762579cm)},
	scale=7.000000,
	back/.style={loosely dotted, thin},
	edge/.style={color=blue!35!black, thick},
	facet/.style={fill=blue!35,fill opacity=0.700000},
	vertex/.style={inner sep=1pt,circle,draw=green!25!black,fill=green!75!black,thick,anchor=base}]

\coordinate (0.90000, 0.30000, 1.70000) at (0.90000, 0.30000, 1.70000);
\coordinate (0.90000, 0.30000, 1.40000) at (0.90000, 0.30000, 1.40000);
\coordinate (0.90000, 0.40000, 1.40000) at (0.90000, 0.40000, 1.40000);
\coordinate (1.20000, 0.40000, 1.40000) at (1.20000, 0.40000, 1.40000);
\coordinate (1.30000, 0.30000, 1.40000) at (1.30000, 0.30000, 1.40000);
\coordinate (1.00000, 0.30000, 1.70000) at (1.00000, 0.30000, 1.70000);
\coordinate (0.90000, 0.40000, 1.70000) at (0.90000, 0.40000, 1.70000);
\draw[edge,back] (0.90000, 0.30000, 1.70000) -- (0.90000, 0.30000, 1.40000);
\draw[edge,back] (0.90000, 0.30000, 1.40000) -- (0.90000, 0.40000, 1.40000);
\draw[edge,back] (0.90000, 0.30000, 1.40000) -- (1.30000, 0.30000, 1.40000);
\node[vertex] at (0.90000, 0.30000, 1.40000)     {};
\fill[facet] (0.90000, 0.40000, 1.70000) -- (0.90000, 0.40000, 1.40000) -- (1.20000, 0.40000, 1.40000) -- cycle {};
\fill[facet] (0.90000, 0.40000, 1.70000) -- (0.90000, 0.30000, 1.70000) -- (1.00000, 0.30000, 1.70000) -- cycle {};
\fill[facet] (0.90000, 0.40000, 1.70000) -- (1.20000, 0.40000, 1.40000) -- (1.30000, 0.30000, 1.40000) -- (1.00000, 0.30000, 1.70000) -- cycle {};
\draw[edge] (0.90000, 0.30000, 1.70000) -- (1.00000, 0.30000, 1.70000);
\draw[edge] (0.90000, 0.30000, 1.70000) -- (0.90000, 0.40000, 1.70000);
\draw[edge] (0.90000, 0.40000, 1.40000) -- (1.20000, 0.40000, 1.40000);
\draw[edge] (0.90000, 0.40000, 1.40000) -- (0.90000, 0.40000, 1.70000);
\draw[edge] (1.20000, 0.40000, 1.40000) -- (1.30000, 0.30000, 1.40000);
\draw[edge] (1.20000, 0.40000, 1.40000) -- (0.90000, 0.40000, 1.70000);
\draw[edge] (1.30000, 0.30000, 1.40000) -- (1.00000, 0.30000, 1.70000);
\draw[edge] (1.00000, 0.30000, 1.70000) -- (0.90000, 0.40000, 1.70000);
\node[vertex] at (0.90000, 0.30000, 1.70000)     {};
\node[vertex] at (0.90000, 0.40000, 1.40000)     {};
\node[vertex] at (1.20000, 0.40000, 1.40000)     {};
\node[vertex] at (1.30000, 0.30000, 1.40000)     {};
\node[vertex] at (1.00000, 0.30000, 1.70000)     {};
\node[vertex] at (0.90000, 0.40000, 1.70000)     {};
\node[vertex,color=red,label=90:${\pp^{(1)}\hh=\hh\hat{\pp}^{(1)}}$] at (1,0.4,1 )  {};
\node[vertex,color=red,label=45:$\pp^{(2)}$] at ( 0.75 ,  0.4  ,  1.4)  {};
\node[vertex,color=red,label=0:${\pp^{(3)}\hh\hh=\hh\hat{\pp}^{(4)}}$] at ( 0.9 ,  0.4 ,  1.4 )  {};

\node[vertex,color=yellow,label=90:$\hat{\pp}^{(2)}$] at (0.81 ,  0.4  ,  1.4 )  {};
\node[vertex,color=yellow,label=-90:$\hat{\pp}^{(3)}$] at ( 0.91 ,  0.4  ,  1.38 )  {};
\end{tikzpicture}
 }
            \subfloat[Solutions and cuts (projected) obtained with \Cref{algo:confidOptimDisag}]{
            \label{fig:example-hoffman}
\begin{tikzpicture}%
	[x={(-0.688249cm, -0.056987cm)},
	y={(-0.723341cm, 0.130317cm)},
	z={(0.055608cm, 0.989833cm)},
	scale=5.000000,
	back/.style={loosely dotted, thin},
	edge/.style={color=blue!95!black, thick},
	edger/.style={color=blue!95!black, dotted},
	facet/.style={fill=blue!35,fill opacity=0.800000},
	vertex/.style={inner sep=0.51pt,circle,draw=black,fill=black,thick,anchor=base},
	sol/.style={inner sep=1pt,circle,draw=blue,fill=blue,thick,anchor=base}]
\draw[color=black,thick,->]   (0.5 , 0.2 ,  1 ) -- (0.7 , 0.2 ,  1 ) node[anchor= north]{$p_1$};
\draw[color=black,thick,->]  (0.5 , 0.2 ,  1 ) -- (0.5 , 0.4 ,  1 ) node[anchor=south ]{$p_2$};
\draw[color=black,thick,->] (0.5 , 0.2 ,  1 )  -- (0.5 , 0.2 ,  1.2 ) node[anchor=south]{$p_3$};

\coordinate (1.40000, 0.20000, 1.40000) at (1.40000, 0.20000, 1.40000);
\coordinate (1.40000, 0.50000, 1.40000) at (1.40000, 0.50000, 1.40000);
\coordinate (0.80000, 0.20000, 1.40000) at (0.80000, 0.20000, 1.40000);
\coordinate (0.80000, 0.50000, 1.40000) at (0.80000, 0.50000, 1.40000);
\coordinate (0.80000, 0.50000, 1.80000) at (0.80000, 0.50000, 1.80000);
\coordinate (0.80000, 0.20000, 1.80000) at (0.80000, 0.20000, 1.80000);
\coordinate (1.40000, 0.50000, 1.80000) at (1.40000, 0.50000, 1.80000);
\coordinate (1.40000, 0.20000, 1.80000) at (1.40000, 0.20000, 1.80000);
\draw[edge,back] (1.40000, 0.20000, 1.40000) -- (1.40000, 0.20000, 1.80000);
\draw[edge,back] (0.80000, 0.20000, 1.80000) -- (1.40000, 0.20000, 1.80000);
\draw[edge,back] (1.40000, 0.50000, 1.80000) -- (1.40000, 0.20000, 1.80000);
\node[vertex] at (1.40000, 0.20000, 1.80000)     {};
\fill[facet,red!60] (0.80000, 0.50000, 1.40000) -- (1.40000, 0.50000, 1.40000) -- (1.40000, 0.20000, 1.40000) -- (0.80000, 0.20000, 1.40000) -- cycle {};
\fill[facet] (0.80000, 0.20000, 1.80000) -- (0.80000, 0.20000, 1.40000) -- (0.80000, 0.50000, 1.40000) -- (0.80000, 0.50000, 1.80000) -- cycle {};
\fill[facet] (1.40000, 0.50000, 1.80000) -- (1.40000, 0.50000, 1.40000) -- (0.80000, 0.50000, 1.40000) -- (0.80000, 0.50000, 1.80000) -- cycle {};
\draw[edger] (1.40000, 0.20000, 1.40000) -- (1.40000, 0.50000, 1.40000);
\draw[edger] (1.40000, 0.20000, 1.40000) -- (0.80000, 0.20000, 1.40000);
\draw[edger] (1.40000, 0.50000, 1.40000) -- (0.80000, 0.50000, 1.40000);
\draw[edger] (1.40000, 0.50000, 1.40000) -- (1.40000, 0.50000, 1.80000);
\draw[edger] (0.80000, 0.20000, 1.40000) -- (0.80000, 0.50000, 1.40000);
\draw[edger] (0.80000, 0.20000, 1.40000) -- (0.80000, 0.20000, 1.80000);
\draw[edger] (0.80000, 0.50000, 1.40000) -- (0.80000, 0.50000, 1.80000);
\draw[edger] (0.80000, 0.50000, 1.80000) -- (0.80000, 0.20000, 1.80000);
\draw[edger] (0.80000, 0.50000, 1.80000) -- (1.40000, 0.50000, 1.80000);
\node[vertex] at (1.40000, 0.20000, 1.40000)     {};
\node[vertex] at (1.40000, 0.50000, 1.40000)     {};
\node[vertex] at (0.80000, 0.20000, 1.40000)     {};
\node[vertex] at (0.80000, 0.50000, 1.40000)     {};
\node[vertex] at (0.80000, 0.50000, 1.80000)     {};
\node[vertex] at (0.80000, 0.20000, 1.80000)     {};
\node[vertex] at (1.40000, 0.50000, 1.80000)     {};
\node[sol,label=90:$\pp^{(1)}$] at (1,0.4,1 )  {};
\node[sol,label=45:$\pp^{(2)}$] at ( 0.75 ,  0.4  ,  1.4)  {};

\end{tikzpicture}
\begin{tikzpicture}%
	[x={(-0.690296cm, 0.111443cm)},
	y={(-0.675921cm, 0.253154cm)},
	z={(0.258111cm, 0.960986cm)},
		scale=5.000000,
		rotate=10,
	back/.style={loosely dotted, thin},
	edge/.style={color=blue!95!black, },
	edger/.style={color=blue!95!black, dotted},
	facet/.style={fill=blue!35,fill opacity=0.800000},
	vertex/.style={inner sep=0.51pt,circle,draw=black,fill=black,thick,anchor=base},
	sol/.style={inner sep=1pt,circle,draw=blue,fill=blue,thick,anchor=base}]
\draw[color=black,thick,->]   (0.5 , 0.2 ,  1 ) -- (0.7 , 0.2 ,  1 ) node[anchor= north]{$p_1$};
\draw[color=black,thick,->]  (0.5 , 0.2 ,  1 ) -- (0.5 , 0.4 ,  1 ) node[anchor=south ]{$p_2$};
\draw[color=black,thick,->] (0.5 , 0.2 ,  1 )  -- (0.5 , 0.2 ,  1.2 ) node[anchor=south]{$p_3$};

\coordinate (1.40000, 0.20000, 1.40000) at (1.40000, 0.20000, 1.40000);
\coordinate (1.40000, 0.50000, 1.40000) at (1.40000, 0.50000, 1.40000);
\coordinate (0.90000, 0.20000, 1.40000) at (0.90000, 0.20000, 1.40000);
\coordinate (0.90000, 0.50000, 1.40000) at (0.90000, 0.50000, 1.40000);
\coordinate (0.90000, 0.50000, 1.80000) at (0.90000, 0.50000, 1.80000);
\coordinate (0.90000, 0.20000, 1.80000) at (0.90000, 0.20000, 1.80000);
\coordinate (1.40000, 0.50000, 1.80000) at (1.40000, 0.50000, 1.80000);
\coordinate (1.40000, 0.20000, 1.80000) at (1.40000, 0.20000, 1.80000);
\draw[edge,back] (1.40000, 0.20000, 1.40000) -- (1.40000, 0.20000, 1.80000);
\draw[edge,back] (0.90000, 0.20000, 1.80000) -- (1.40000, 0.20000, 1.80000);
\draw[edge,back] (1.40000, 0.50000, 1.80000) -- (1.40000, 0.20000, 1.80000);
\node[vertex] at (1.40000, 0.20000, 1.80000)     {};
\fill[facet,orange!30] (0.90000, 0.50000, 1.40000) -- (1.40000, 0.50000, 1.40000) -- (1.40000, 0.20000, 1.40000) -- (0.90000, 0.20000, 1.40000) -- cycle {};
\fill[facet,red!60] (0.90000, 0.20000, 1.80000) -- (0.90000, 0.20000, 1.40000) -- (0.90000, 0.50000, 1.40000) -- (0.90000, 0.50000, 1.80000) -- cycle {};
\fill[facet] (1.40000, 0.50000, 1.80000) -- (1.40000, 0.50000, 1.40000) -- (0.90000, 0.50000, 1.40000) -- (0.90000, 0.50000, 1.80000) -- cycle {};
\draw[edger] (1.40000, 0.20000, 1.40000) -- (1.40000, 0.50000, 1.40000);
\draw[edger] (1.40000, 0.20000, 1.40000) -- (0.90000, 0.20000, 1.40000);
\draw[edger] (1.40000, 0.50000, 1.40000) -- (0.90000, 0.50000, 1.40000);
\draw[edger] (1.40000, 0.50000, 1.40000) -- (1.40000, 0.50000, 1.80000);
\draw[edge] (0.90000, 0.20000, 1.40000) -- (0.90000, 0.50000, 1.40000);
\draw[edger] (0.90000, 0.20000, 1.40000) -- (0.90000, 0.20000, 1.80000);
\draw[edger] (0.90000, 0.50000, 1.40000) -- (0.90000, 0.50000, 1.80000);
\draw[edger] (0.90000, 0.50000, 1.80000) -- (0.90000, 0.20000, 1.80000);
\draw[edger] (0.90000, 0.50000, 1.80000) -- (1.40000, 0.50000, 1.80000);
\node[vertex] at (1.40000, 0.20000, 1.40000)     {};
\node[vertex] at (1.40000, 0.50000, 1.40000)     {};
\node[vertex] at (0.90000, 0.20000, 1.40000)     {};
\node[vertex] at (0.90000, 0.50000, 1.40000)     {};
\node[vertex] at (0.90000, 0.50000, 1.80000)     {};
\node[vertex] at (0.90000, 0.20000, 1.80000)     {};
\node[vertex] at (1.40000, 0.50000, 1.80000)     {};
\node[sol,label=90:$\pp^{(1)}$] at (1,0.4,1 )  {};
\node[sol,label=45:$\pp^{(2)}$] at ( 0.75 ,  0.4  ,  1.4)  {};
\node[sol,label=-40:$\pp^{(3)}$] at ( 0.9 ,  0.4 ,  1.4 )  {};

\end{tikzpicture} %
 }

\subfloat[Solutions and cuts (projected) obtained with \Cref{algo:decompAPM-polyhedral}]{      
\label{fig:example-genAlgo}
\begin{tikzpicture}%
	[x={(-0.688249cm, -0.056987cm)},
	y={(-0.723341cm, 0.130317cm)},
	z={(0.055608cm, 0.989833cm)},
	scale=5.000000,
	back/.style={loosely dotted, thin},
	edge/.style={color=blue!95!black, thick},
	edger/.style={color=blue!95!black, dotted},
	facet/.style={fill=blue!35,fill opacity=0.800000},
	vertex/.style={inner sep=0.51pt,circle,draw=black,fill=black,thick,anchor=base},
	sol/.style={inner sep=1pt,circle,draw=blue,fill=blue,thick,anchor=base}]
\draw[color=black,thick,->]   (0.5 , 0.2 ,  1 ) -- (0.7 , 0.2 ,  1 ) node[anchor= north]{$p_1$};
\draw[color=black,thick,->]  (0.5 , 0.2 ,  1 ) -- (0.5 , 0.4 ,  1 ) node[anchor=south ]{$p_2$};
\draw[color=black,thick,->] (0.5 , 0.2 ,  1 )  -- (0.5 , 0.2 ,  1.2 ) node[anchor=south]{$p_3$};
\coordinate (1.40000, 0.50000, 1.30000) at (1.40000, 0.50000, 1.30000);
\coordinate (1.40000, 0.40000, 1.30000) at (1.40000, 0.40000, 1.30000);
\coordinate (1.40000, 0.20000, 1.33333) at (1.40000, 0.20000, 1.33333);
\coordinate (1.30000, 0.50000, 1.30000) at (1.30000, 0.50000, 1.30000);
\coordinate (0.80000, 0.50000, 1.38333) at (0.80000, 0.50000, 1.38333);
\coordinate (0.80000, 0.20000, 1.43333) at (0.80000, 0.20000, 1.43333);
\coordinate (0.80000, 0.20000, 1.80000) at (0.80000, 0.20000, 1.80000);
\coordinate (0.80000, 0.50000, 1.80000) at (0.80000, 0.50000, 1.80000);
\coordinate (1.40000, 0.20000, 1.80000) at (1.40000, 0.20000, 1.80000);
\coordinate (1.40000, 0.50000, 1.80000) at (1.40000, 0.50000, 1.80000);
\draw[edge,back] (1.40000, 0.20000, 1.33333) -- (1.40000, 0.20000, 1.80000);
\draw[edge,back] (0.80000, 0.20000, 1.80000) -- (1.40000, 0.20000, 1.80000);
\draw[edge,back] (1.40000, 0.20000, 1.80000) -- (1.40000, 0.50000, 1.80000);
\node[vertex] at (1.40000, 0.20000, 1.80000)     {};
\fill[facet,red!60] (0.80000, 0.20000, 1.43333) -- (1.40000, 0.20000, 1.33333) -- (1.40000, 0.40000, 1.30000) -- (1.30000, 0.50000, 1.30000) -- (0.80000, 0.50000, 1.38333) -- cycle {};
\fill[facet] (0.80000, 0.50000, 1.80000) -- (0.80000, 0.50000, 1.38333) -- (0.80000, 0.20000, 1.43333) -- (0.80000, 0.20000, 1.80000) -- cycle {};
\fill[facet] (1.40000, 0.50000, 1.80000) -- (1.40000, 0.50000, 1.30000) -- (1.30000, 0.50000, 1.30000) -- (0.80000, 0.50000, 1.38333) -- (0.80000, 0.50000, 1.80000) -- cycle {};
\fill[facet] (1.30000, 0.50000, 1.30000) -- (1.40000, 0.50000, 1.30000) -- (1.40000, 0.40000, 1.30000) -- cycle {};
\draw[edger] (1.40000, 0.50000, 1.30000) -- (1.40000, 0.40000, 1.30000);
\draw[edger] (1.40000, 0.50000, 1.30000) -- (1.30000, 0.50000, 1.30000);
\draw[edger] (1.40000, 0.50000, 1.30000) -- (1.40000, 0.50000, 1.80000);
\draw[edger] (1.40000, 0.40000, 1.30000) -- (1.40000, 0.20000, 1.33333);
\draw[edger] (1.40000, 0.40000, 1.30000) -- (1.30000, 0.50000, 1.30000);
\draw[edger] (1.40000, 0.20000, 1.33333) -- (0.80000, 0.20000, 1.43333);
\draw[edger] (1.30000, 0.50000, 1.30000) -- (0.80000, 0.50000, 1.38333);
\draw[edger] (0.80000, 0.50000, 1.38333) -- (0.80000, 0.20000, 1.43333);
\draw[edger] (0.80000, 0.50000, 1.38333) -- (0.80000, 0.50000, 1.80000);
\draw[edger] (0.80000, 0.20000, 1.43333) -- (0.80000, 0.20000, 1.80000);
\draw[edger] (0.80000, 0.20000, 1.80000) -- (0.80000, 0.50000, 1.80000);
\draw[edger] (0.80000, 0.50000, 1.80000) -- (1.40000, 0.50000, 1.80000);
\node[vertex] at (1.40000, 0.50000, 1.30000)     {};
\node[vertex] at (1.40000, 0.40000, 1.30000)     {};
\node[vertex] at (1.40000, 0.20000, 1.33333)     {};
\node[vertex] at (1.30000, 0.50000, 1.30000)     {};
\node[vertex] at (0.80000, 0.50000, 1.38333)     {};
\node[vertex] at (0.80000, 0.20000, 1.43333)     {};
\node[vertex] at (0.80000, 0.20000, 1.80000)     {};
\node[vertex] at (0.80000, 0.50000, 1.80000)     {};
\node[vertex] at (1.40000, 0.50000, 1.80000)     {};

\node[sol,label=-90:$\hat{\pp}^{(1)}$] at (1.0 , 0.4 ,  1 )  {};
\node[sol,label=70:$\hat{\pp}^{(2)}$] at (0.81 ,  0.4  ,  1.4 )  {};
\end{tikzpicture} %
\begin{tikzpicture}%
	[x={(-0.690296cm, 0.111443cm)},
	y={(-0.675921cm, 0.253154cm)},
	z={(0.258111cm, 0.960986cm)},
	scale=5.000000,
	back/.style={loosely dotted, thin},
	edge/.style={color=blue!95!black, },
	edger/.style={color=blue!95!black, dotted},
	facet/.style={fill=blue!35,fill opacity=0.800000},
	vertex/.style={inner sep=0.51pt,circle,draw=black,fill=black,thick,anchor=base},
	sol/.style={inner sep=1pt,circle,draw=blue,fill=blue,thick,anchor=base}]
\draw[color=black,thick,->]   (0.5 , 0.2 ,  1 ) -- (0.7 , 0.2 ,  1 ) node[anchor= north]{$p_1$};
\draw[color=black,thick,->]  (0.5 , 0.2 ,  1 ) -- (0.5 , 0.4 ,  1 ) node[anchor=south ]{$p_2$};
\draw[color=black,thick,->] (0.5 , 0.2 ,  1 )  -- (0.5 , 0.2 ,  1.2 ) node[anchor=south]{$p_3$};
\coordinate (1.40000, 0.50000, 1.30000) at (1.40000, 0.50000, 1.30000);
\coordinate (1.40000, 0.40000, 1.30000) at (1.40000, 0.40000, 1.30000);
\coordinate (1.40000, 0.20000, 1.33333) at (1.40000, 0.20000, 1.33333);
\coordinate (0.80000, 0.50000, 1.68653) at (0.80000, 0.50000, 1.68653);
\coordinate (1.30000, 0.50000, 1.30000) at (1.30000, 0.50000, 1.30000);
\coordinate (0.91243, 0.50000, 1.36459) at (0.91243, 0.50000, 1.36459);
\coordinate (0.80000, 0.20000, 1.68653) at (0.80000, 0.20000, 1.68653);
\coordinate (0.89389, 0.20000, 1.41769) at (0.89389, 0.20000, 1.41769);
\coordinate (0.80000, 0.20000, 1.80000) at (0.80000, 0.20000, 1.80000);
\coordinate (0.80000, 0.50000, 1.80000) at (0.80000, 0.50000, 1.80000);
\coordinate (1.40000, 0.20000, 1.80000) at (1.40000, 0.20000, 1.80000);
\coordinate (1.40000, 0.50000, 1.80000) at (1.40000, 0.50000, 1.80000);
\draw[edge,back] (1.40000, 0.20000, 1.33333) -- (1.40000, 0.20000, 1.80000);
\draw[edge,back] (0.80000, 0.20000, 1.80000) -- (1.40000, 0.20000, 1.80000);
\draw[edge,back] (1.40000, 0.20000, 1.80000) -- (1.40000, 0.50000, 1.80000);
\node[vertex] at (1.40000, 0.20000, 1.80000)     {};
\fill[facet,orange!30] (0.89389, 0.20000, 1.41769) -- (1.40000, 0.20000, 1.33333) -- (1.40000, 0.40000, 1.30000) -- (1.30000, 0.50000, 1.30000) -- (0.91243, 0.50000, 1.36459) -- cycle {};
\fill[facet,red!60] (0.89389, 0.20000, 1.41769) -- (0.91243, 0.50000, 1.36459) -- (0.80000, 0.50000, 1.68653) -- (0.80000, 0.20000, 1.68653) -- cycle {};
\fill[facet] (0.80000, 0.50000, 1.80000) -- (0.80000, 0.50000, 1.68653) -- (0.80000, 0.20000, 1.68653) -- (0.80000, 0.20000, 1.80000) -- cycle {};
\fill[facet] (1.40000, 0.50000, 1.80000) -- (1.40000, 0.50000, 1.30000) -- (1.30000, 0.50000, 1.30000) -- (0.91243, 0.50000, 1.36459) -- (0.80000, 0.50000, 1.68653) -- (0.80000, 0.50000, 1.80000) -- cycle {};
\fill[facet] (1.30000, 0.50000, 1.30000) -- (1.40000, 0.50000, 1.30000) -- (1.40000, 0.40000, 1.30000) -- cycle {};
\draw[edger] (1.40000, 0.50000, 1.30000) -- (1.40000, 0.40000, 1.30000);
\draw[edger] (1.40000, 0.50000, 1.30000) -- (1.30000, 0.50000, 1.30000);
\draw[edger] (1.40000, 0.50000, 1.30000) -- (1.40000, 0.50000, 1.80000);
\draw[edger] (1.40000, 0.40000, 1.30000) -- (1.40000, 0.20000, 1.33333);
\draw[edger] (1.40000, 0.40000, 1.30000) -- (1.30000, 0.50000, 1.30000);
\draw[edger] (1.40000, 0.20000, 1.33333) -- (0.89389, 0.20000, 1.41769);
\draw[edger] (0.80000, 0.50000, 1.68653) -- (0.91243, 0.50000, 1.36459);
\draw[edger] (0.80000, 0.50000, 1.68653) -- (0.80000, 0.20000, 1.68653);
\draw[edger] (0.80000, 0.50000, 1.68653) -- (0.80000, 0.50000, 1.80000);
\draw[edger] (1.30000, 0.50000, 1.30000) -- (0.91243, 0.50000, 1.36459);
\draw[edge] (0.91243, 0.50000, 1.36459) -- (0.89389, 0.20000, 1.41769);
\draw[edger] (0.80000, 0.20000, 1.68653) -- (0.89389, 0.20000, 1.41769);
\draw[edger] (0.80000, 0.20000, 1.68653) -- (0.80000, 0.20000, 1.80000);
\draw[edger] (0.80000, 0.20000, 1.80000) -- (0.80000, 0.50000, 1.80000);
\draw[edger] (0.80000, 0.50000, 1.80000) -- (1.40000, 0.50000, 1.80000);
\node[vertex] at (1.40000, 0.50000, 1.30000)     {};
\node[vertex] at (1.40000, 0.40000, 1.30000)     {};
\node[vertex] at (1.40000, 0.20000, 1.33333)     {};
\node[vertex] at (0.80000, 0.50000, 1.68653)     {};
\node[vertex] at (1.30000, 0.50000, 1.30000)     {};
\node[vertex] at (0.91243, 0.50000, 1.36459)     {};
\node[vertex] at (0.80000, 0.20000, 1.68653)     {};
\node[vertex] at (0.89389, 0.20000, 1.41769)     {};
\node[vertex] at (0.80000, 0.20000, 1.80000)     {};
\node[vertex] at (0.80000, 0.50000, 1.80000)     {};
\node[vertex] at (1.40000, 0.50000, 1.80000)     {};

\node[sol,label=-90:$\hat{\pp}^{(1)}$] at (1.0 , 0.4 ,  1 )  {};
\node[sol,label=70:$\hat{\pp}^{(2)}$] at (0.81 ,  0.4  ,  1.4 )  {};
\node[sol,label=-90:$\hat{\pp}^{(3)}$] at ( 0.91 ,  0.4  ,  1.38 )  {};
\end{tikzpicture}
\begin{tikzpicture}%
	[x={(-0.539421cm, -0.075235cm)},
	y={(-0.822314cm, 0.261371cm)},
	z={(0.181175cm, 0.962302cm)},
	scale=5.000000,
	back/.style={loosely dotted, thin},
	edge/.style={color=blue!95!black, },
	edger/.style={color=blue!95!black, dotted},
	facet/.style={fill=blue!35,fill opacity=0.800000},
	vertex/.style={inner sep=0.51pt,circle,draw=black,fill=black,thick,anchor=base},
	sol/.style={inner sep=1pt,circle,draw=blue,fill=blue,thick,anchor=base}]

\draw[color=black,thick,->]   (0.5 , 0.2 ,  1 ) -- (0.7 , 0.2 ,  1 ) node[anchor= north]{$p_1$};
\draw[color=black,thick,->]  (0.5 , 0.2 ,  1 ) -- (0.5 , 0.4 ,  1 ) node[anchor=north ]{$p_2$};
\draw[color=black,thick,->] (0.5 , 0.2 ,  1 )  -- (0.5 , 0.2 ,  1.2 ) node[anchor=south]{$p_3$};
\coordinate (1.40000, -0.20000, 1.40000) at (1.40000, -0.20000, 1.40000);
\coordinate (1.40000, 0.50000, 1.40000) at (1.40000, 0.50000, 1.40000);
\coordinate (0.90007, 0.50000, 1.40000) at (0.90007, 0.50000, 1.40000);
\coordinate (0.86917, -0.20000, 1.48847) at (0.86917, -0.20000, 1.48847);
\coordinate (0.90007, 0.29993, 1.40000) at (0.90007, 0.29993, 1.40000);
\coordinate (0.76037, -0.20000, 1.80000) at (0.76037, -0.20000, 1.80000);
\coordinate (0.76037, 0.50000, 1.80000) at (0.76037, 0.50000, 1.80000);
\coordinate (1.40000, 0.50000, 1.80000) at (1.40000, 0.50000, 1.80000);
\coordinate (1.40000, -0.20000, 1.80000) at (1.40000, -0.20000, 1.80000);
\draw[edge,back] (1.40000, -0.20000, 1.40000) -- (1.40000, -0.20000, 1.80000);
\draw[edge,back] (0.76037, -0.20000, 1.80000) -- (1.40000, -0.20000, 1.80000);
\draw[edge,back] (1.40000, 0.50000, 1.80000) -- (1.40000, -0.20000, 1.80000);
\node[vertex] at (1.40000, -0.20000, 1.80000)     {};
\fill[facet,orange!30] (0.90007, 0.29993, 1.40000) -- (1.40000, -0.20000, 1.40000) -- (0.86917, -0.20000, 1.48847) -- cycle {};
\fill[facet,orange!30] (0.76037, 0.50000, 1.80000) -- (0.90007, 0.50000, 1.40000) -- (0.90007, 0.29993, 1.40000) -- (0.86917, -0.20000, 1.48847) -- (0.76037, -0.20000, 1.80000) -- cycle {};
\fill[facet,red!60] (0.90007, 0.29993, 1.40000) -- (1.40000, -0.20000, 1.40000) -- (1.40000, 0.50000, 1.40000) -- (0.90007, 0.50000, 1.40000) -- cycle {};
\fill[facet] (1.40000, 0.50000, 1.80000) -- (1.40000, 0.50000, 1.40000) -- (0.90007, 0.50000, 1.40000) -- (0.76037, 0.50000, 1.80000) -- cycle {};
\draw[edger] (1.40000, -0.20000, 1.40000) -- (1.40000, 0.50000, 1.40000);
\draw[edger] (1.40000, -0.20000, 1.40000) -- (0.86917, -0.20000, 1.48847);
\draw[edge] (1.40000, -0.20000, 1.40000) -- (0.90007, 0.29993, 1.40000);
\draw[edger] (1.40000, 0.50000, 1.40000) -- (0.90007, 0.50000, 1.40000);
\draw[edger] (1.40000, 0.50000, 1.40000) -- (1.40000, 0.50000, 1.80000);
\draw[edge] (0.90007, 0.50000, 1.40000) -- (0.90007, 0.29993, 1.40000);
\draw[edger] (0.90007, 0.50000, 1.40000) -- (0.76037, 0.50000, 1.80000);
\draw[edge] (0.86917, -0.20000, 1.48847) -- (0.90007, 0.29993, 1.40000);
\draw[edger] (0.86917, -0.20000, 1.48847) -- (0.76037, -0.20000, 1.80000);
\draw[edger] (0.76037, -0.20000, 1.80000) -- (0.76037, 0.50000, 1.80000);
\draw[edger] (0.76037, 0.50000, 1.80000) -- (1.40000, 0.50000, 1.80000);
\node[vertex] at (1.40000, -0.20000, 1.40000)     {};
\node[vertex] at (1.40000, 0.50000, 1.40000)     {};
\node[vertex] at (0.90007, 0.50000, 1.40000)     {};
\node[vertex] at (0.86917, -0.20000, 1.48847)     {};
\node[vertex] at (0.90007, 0.29993, 1.40000)     {};
\node[vertex] at (0.76037, -0.20000, 1.80000)     {};
\node[vertex] at (0.76037, 0.50000, 1.80000)     {};
\node[vertex] at (1.40000, 0.50000, 1.80000)     {};

\node[sol,label=-90:$\hat{\pp}^{(1)}$] at (1.0 , 0.4 ,  1 )  {};
\node[sol,label=20:$\hat{\pp}^{(2)}$] at (0.81 ,  0.4  ,  1.4 )  {};
\node[sol,label=-90:$\hat{\pp}^{(3)}$] at ( 0.91 ,  0.4  ,  1.38 )  {};
\node[sol,label=90:${\hat{\pp}^{(4)}}$] at ( 0.9 ,  0.4 ,  1.4 )  {};
\end{tikzpicture}
}
    \caption{Illustration of the iterations of the proposed decomposition method. \textit{The cut $p_3\geq 1.4$, which is added at first for   \Cref{algo:confidOptimDisag} , is only added at the third iteration of  \Cref{algo:decompAPM-polyhedral} } \vspace{-0.5cm}}
  \label{fig:exampleT4Algos}
\end{figure*}
The 4 successive solutions and the 3 added cuts are represented in the three dimensional space on \Cref{fig:example-genAlgo}: we observe that the last cut needed to obtain the convergence of \Cref{algo:decompAPM-polyhedral}, corresponds to the first one added with \Cref{algo:confidOptimDisag}.

\subsection{A Nonconvex Example: Management of a Microgrid}
\label{subsec:appli-microgrid}
In this section, we illustrate the proposed method on a larger scale practical example from energy. We consider an electricity microgrid \cite{katiraei2008microgrids}  composed of N electricity consumers with flexible appliances (such as  electric vehicles or water heaters), a photovoltaic (PV) power plant and a conventional generator. 
The operator of the microgrid aims at satisfying the demand constraints of consumers over a set of time periods $\T=\{1,\dots,T\}$, while minimizing the energy cost for the community.
We have the following characteristics:
\begin{itemize}[wide]
\item[] -- the PV plant generates a nondispatchable power profile $(\ppv_t)_{t\in\Tall}$ at marginal cost zero;
\item[] -- the conventional generator has a starting cost $\Cst$, minimal and maximal power production $\pglb,\pgub$, and piecewise-linear and continuous generation cost function $\pg \mapsto f(\pg)$:
\begin{equation*}
  f(\pg)= \alpha_k + c_k \pg   , \text{ if } \pg \in \I_k \eqd [\theta_{k-1}, \ \theta_{k} [, \ k=1\dots K ,
      \;\text{where }\theta_0\eqd 0 \text{ and }\theta_K\eqd \pgub;
\end{equation*}
\item[] -- each agent $n\in\Nall$ has some flexible appliances which require a global energy demand $E_n$ on $\Tall$, and has consumption constraints on the total household consumption, on each time period $t\in\Tall$, that are formulated with $\llbx_n,\uubx_n$. These parameters are confidential because they could for instance contain some information on agent $n$ habits.
\end{itemize}
The master problem \eqref{pb:master} can be written as the following MILP \eqref{modelMG}:
\allowdisplaybreaks
\begin{subequations} \label{modelMG}
\begin{align}
& \hspace{-9pt} \min_{\pp,\ppg, (\ppg_k),(\bm{\bo}_k),\bm{\xon}, \bm{\xst} }  \sum_{t\in\Tall} \Big( \alpha_1 \xon_t+ \sum_k c_k {\pg_k}_t + \Cst \xst_t \Big) \label{eq:MGobj}\\
\label{eq:formPcwGenCost1} &\pg_t= \txt\sum_{k=1}^K \pg_{k,t} ,  \ \forall t \in \Tall \\
\label{eq:formPcwGenCost2}  & \bo_{k,t}(\theta_k- \theta_{k-1}) \leq  \pg_{k,t} \leq \bo_{k-1,t}(\theta_k- \theta_{k-1}) ,  \ \forall 1 \leq k \leq K, \ \forall t \in \Tall \\
\label{eq:consStart}& \xst_t \geq  \xon_{t}-\xon_{t-1}, \ \forall t \in \{2,\dots,T\}\\ 
\label{eq:consOn}& \pglb \xon_t \leq \pg_t \leq \pgub \xon_t  ,  \ \forall t \in \T\\
& \xon_t,\xst_t, \bo_{1,t}, \dots, \bo_{K-1,t} \in \{0,1\} ,  \ \forall t \in \Tall \\
\label{eq:consProdTot}& \pp\leq \pppv + \ppg \\
\label{eq:consProdDemand}& \pp^\tr \dsoneT= \EE^\tr \dsoneN \\
\label{eq:consProdBounds}& \llbx^\tr \dsoneN \leq \pp \leq \uubx^\tr  \dsoneN \ .
\end{align}
\end{subequations}
In this formulation (\ref{eq:formPcwGenCost1}-\ref{eq:formPcwGenCost2}), where $ \bo_{0,t} \eqd 1$ and $\bo_{K,t}\eqd 0$, are a mixed integer formulation of the generation cost function $f$. One can show that the Boolean variable $\bo_{k,t}$ is equal to one \textit{iff} $\pg_{t} \geq \theta_k$ for each $k\in \{1,\dots,K-1\}$. Note that only $\alpha_1$ appears in \eqref{eq:MGobj} because of the continuity of  $f$.

Constraints (\ref{eq:consStart}-\ref{eq:consOn}) ensure the on/off and starting constraints of the power plant, \eqref{eq:consProdTot} ensures that the power allocated to   consumption is not above the total production, and (\ref{eq:consProdDemand}-\ref{eq:consProdBounds}) are the aggregated feasibility conditions already referred to in \eqref{eq:agg-cond}.
The nonconvexity of \eqref{modelMG} comes from the existence of starting costs and constraints of minimal power, which makes necessary to use Boolean state variables $\xst,\xon$.

We simulate the problem described above for different values of $N\in \{ 2^4$,$2^5$, $2^6$,$2^7$,$2^8 \}$ and one hundred  instances with random parameters for each value of $N$. A scaling factor $\scal_N = N/20$ is applied on parameters to ensure that production capacity  is large enough to meet consumers demand. The parameters are chosen as follows:
\begin{itemize}[wide,labelindent=0pt]
\item[] -- $T=24$ (hours of a day);
\item[] -- production costs: $K=3$ , $\theta=[0,70,100,300]\scal_N, \bm{c}=[0.2,0.4,0.5]$, $\pglb\hh =\hh50\scal_N , \pgub\hh=\hh300\scal_N$, $\alpha_1 \hh =\hh 4$ and $\Cst=15$;
\item[] -- photovoltaic: $\ppv_t\hh\hh=\hh\left[ 50 (  1\hh-\hh\mathrm{cos}(\frac{ (t-6)2\pi}{16})\hh+\hh \U([0,10])\right]\hh\scal_N$ for $t\in \{6,\dots,20\}$, $\ppv_t=0$ otherwise; %
\item[] -- consumption parameters are drawn randomly with: $\lbx\nt \sim \U([0,10])$, $\ubx\nt \sim \U([0,5])+\lbx\nt $ and $E_n \sim  \U([ \dsoneT^\tr \llbx_n, \dsoneT^\tr \uubx_n])$, so that individual feasibility ($\X_n \neq \emptyset$) is ensured.%
\end{itemize}
\renewcommand{\arraystretch}{1.2}
\begin{table}[ht]
\centering
\begin{small}
\begin{tabular}{|l|c|c|c|c| c|}
\hline 
$N=$ & $2^4$           &  $2^5$ &  $2^6$&  $2^7$ & $2^8$ \\\hline \hline 
\#~master &  193.6	   &  194.1 & 225.5 &  210.9  & 194.0  \\\hline 
\#~projs.& 9507     & 15367 &24319 & 26538 & 26646  \\\hline%
\end{tabular}
\end{small}
\caption{number of subproblems solved  (average on 100 instances)\vspace{-0.5cm}}%
\label{tab:NumberSubpb}
\end{table}
We implement \Cref{algo:confidOptimDisag} using Python 3.5. The MILP  \eqref{modelMG} is solved using Cplex Studio 12.6 and  Pyomo interface. Simulations are run on a single core  of a cluster at 3GHz. %
For the convergence criteria (see Lines \ref{algline:NIAPM-conv-APM} and \ref{algline:NIAPM-conv-disag} of \Cref{algo:vonNeumannProjasymConfid}), we use $\epsdis=0.01$ with the operator norm  defined by $\normop{\xx}= \max_{n\in \Nall} \sum_t |x\nt|$ (to avoid the $\sqrt{N}$ factor in the convergence criteria appearing with $\norm{.}_2$), and starts with $\epscvg=0.1$. 
The largest instances took around 10 minutes to be solved  in this configuration and without parallel implementation. As the CPU time needed %
depends on the cluster load, it is not a reliable indicator of the influence on $N$ on the complexity of the problems. 
Moreover, one advantage of the proposed method is that the projections in APM can be computed locally by each agent in parallel, which could not be implemented here for practical reasons.
 
Table~\ref{tab:NumberSubpb}
gives the number of master problems solved and the total number of projections computed, 
 on average over the hundred instance  for each value of $N$.

One observes that the number of master problems \eqref{modelMG} solved %
(number of ``cuts'' added), remains almost constant  when $N$ increases. In all instances, this number is way below the upper bound of $2^{24}> 1,6\times 10^7$ possible constraints (see  \Cref{prop:finiteNumberCuts}), which suggests that only a limited  number of constraints are added in practice.
The average total number of projections computed for each instance (total number of iterations of the \textbf{while} loop of \Cref{algo:vonNeumannProjasymConfid}, Line \ref{algline:NIAPM-mainLoop} over all calls of APM in the instance) increases in a sublinear way
 which is even better that one could expect from the upper bound given in \Cref{prop:boundCvgRate}.

\section{Conclusion}

We provided a non-intrusive algorithm that enables to compute an optimal resource allocation, solution of a--possibly nonconvex--optimization problem, and affect to each agent an individual profile satisfying a global demand and lower and upper bounds constraints. 
Our method uses local projections and works in a distributed fashion. Hence, the  resolution of the problem is still efficient 
even in the case of a very large number of agents.
 The method is also privacy-preserving, as agents do not need to reveal any information on their constraints or their individual profile to a third party.

Several extensions and generalizations can be considered for this work.
\Cref{sec:generArbitraryIndividualSet} generalizes the procedure to arbitrary polyhedral constraints for agents. However, the number of constraints (cuts) added to the master problem is not proved to be finite as done in the transportation case.
Proving that only a finite number of constraints can be added (maybe up to a refinement procedure of the current constraint obtained) will enable to have a termination result for the algorithm in the general polyhedral case. 
In the transportation case, we showed the geometric convergence of APM with a rate linear in the number of agents.
 Moreover, the number of cuts added in the procedure is finite but the upper bound that we have remains exponential.
  In practice however, the number of constraints to consider remains small, as seen in \Cref{sec:examples}.
   A thinner upper bound on the number of cuts added in the algorithm in this case %
would constitute an interesting result.

\section*{Acknowledgments}
  We thank the referees for their comments, leading to improvements
  of this paper. We thank Daniel Augot for very helpful discussions
  concerning secure multiparty computations and cryptographic protocols.
\begin{appendix}

\section{Proof of \texorpdfstring{\Cref{prop:factsOncut}}{}}
\label{app:proofFactsonCuts}

\begin{proof}[Proof of  \Cref{claimFactcut1}]

Let us write the stationarity  conditions associated to problem \eqref{pb:minSquare}:
\begin{align} \label{eq:condoptxE}
\forall n\in \Nall,  \forall t \in \Tall, \quad   0 = (x\nt - y\nt) - \lambda_n - \umu\nt + \omu\nt   \quad  &\text{and } \quad   y\nt = x\nt + \nu_t \ .
\end{align}
By summing the latter equalities on $t\in\Tall$ and $n\in\Nall$, we obtain the
 equalities:
\begin{align}
\label{eq:KKT1factsOnCut}& \sum_t \nu_t = 	\sum_t y\nt - E_n  , \ \forall n\in\Nall\ \quad \quad \quad \quad 
p_t = \sum_n x\nt + N \nu_t , \ \forall t \in \Tall \enspace .
\end{align}
Moreover, by summing over $t\in \Tin_n$ the first series of equalities in~\eqref{eq:condoptxE}, and by using the complementary slackness conditions, which entail that $\umu\nt =\omu\nt=0$ for $t\in\Tin_n$, we get
\begin{align}
\label{eq:KKTlambdafactsOnCut}& | \Tin_n| \lambda_n = E_n - \sum_{t \in \uH_n} \lbx\nt  - \sum_{t \in \Tin_n } y\nt - \sum_{t\in \oH_n} \ubx\nt , \forall n \in \Nall \ ,%
\end{align}
where we define for each $n \in\Nall$:
\begin{equation*}
 \Tin_n \eqDef \{t \ | \  \lbx\nt < x\nt < \ubx\nt\}, \quad \    \uH_n=\{t \ | \ x\nt = \lbx\nt \}  \ \text{  and  } \ \oH_n=\{t \ | \ x\nt = \ubx\nt \} \ . 
\end{equation*}
From \eqref{eq:KKT1factsOnCut} and \Cref{assp:PverifAggregCond} giving the equality $\sum_n E_n = \sum_t p_t$, we obtain $ \sum_t \nu_t =0$ and:
\begin{equation} \label{eq:xpsumE}
\forall n \in \Nall, \ \txt \sum_{t\in\Tall} y\nt = E_n .
\end{equation}
Let us show that $\forall t \in \Tcut, \nu_t >0$. For $t\in\Tcut$, there exists $n$ s.t. $y\nt > \ubx\nt $. From \eqref{eq:condoptxE}  we get:
\begin{equation} \label{eq:posit_nu_t}
\nu_t=y\nt - x\nt \geq y\nt - \ubx\nt > 0 \ .
\end{equation} 
Suppose  by contradiction that there exists $n \notin \Ncut$ 
and $\hatt \in \Tcut$ such that $\x\nth < \ubx\nth$.
We obtain from complementary slackness conditions and the first equality in~\eqref{eq:condoptxE}: 
\begin{equation} \label{eq:lambda_neg_from_contrad}
\x\nth \geq y\nth + \lambda_n = x\nth + \nu_{\hatt} + \lambda_n \ \text{ which implies that } \  \lambda_n <0 \ .
\end{equation}
From this, we obtain $\oH_n \subset \Tcut$: indeed, for $t \in \oH_n $, we have $y\nt+ \lambda_n \geq \ubx\nt$ which gives:  
\begin{equation*}
 y\nt \geq \ubx\nt - \lambda_n > \ubx\nt \ \text{which implies } t \in \Tcut \ .
\end{equation*}
From the condition \eqref{eq:xpsumE} and as $\nu_t >0$  for each $t\in\oH_n$ because $\oH_n \subset \Tcut$ and \eqref{eq:posit_nu_t}, we get:
\begin{align*}
  & 0= \hh \sum_{t\in \Tall}(y\nt-x\nt)= \sum_{t\in\uH_n} \hh (y\nt- \lbx\nt) + \sum_{t\in \Tin_n}\hh (-\lambda_n) + \sum_{t \in \oH_n} \hh\nu_t\enspace \text{i.e.,}\\
  &  \qquad\qquad\qquad\sum_{t\in\uH_n}\hh (y\nt- \lbx\nt) =\hh \sum_{t\in \Tin_n} \hh\lambda_n - \hh\sum_{t \in \oH_n} \hh\nu_t  ,
\end{align*}
which is strictly negative: this implies that there exists $t'\in \uH_n$ such that $y_{n,t'}< \lbx_{n,t'}$. Necessarily, $t' \notin \Tcut$ because $\nu_{t'} = y_{n,t'} - x_{n,t'} < \lbx_{n,t'}-\lbx_{n,t'}=0 $.
 Then, as we have $\sum_{m\in \Nall} y_{m,t'} =p_{t'}\geq \sum_m \lbx_{m,t'} $, there  exists $m\in \Nall$ such that $y_{m,t'}>\lbx_{m,t'}$.
If $\lambda_m \leq 0$, %
and as $x_{m,t'}= \y_{m,t'} -\nu_{t'} > \lbx_{m,t'}  $, we get:
\begin{equation*}
x_{m,t'}= \min(\ubx_{m,t'},  \y_{m,t'} + \lambda_m ) \leq  \y_{m,t'} + \lambda_m \leq \y_{m,t'} = x_{m,t'} + \nu_{t'} < x_{m,t'} \ ,
\end{equation*}
which is impossible, thus $\lambda_m>0$. Now, we observe that $\Tin_n \subset \Tcut$. Indeed, otherwise, if $t"\in \Tin_n \cap \Tcut^c$, we have $\nu_{t"}= - \lambda_n > 0$ and $x_{m,t"} = y_{m,t"} -\nu_{t"} < y_{m,t"}  < \ubx_{m,t"}$, thus we get:
\begin{equation*}
\x_{m,t"} = \max(\lbx_{m,t"}, \y_{m,t"} + \lambda_m) \geq \y_{m,t"} + \lambda_m > \y_{m,t"} = x_{m,t"} +\nu_{t"} + \lambda_m > x_{m,t"} \ ,
\end{equation*} 
which is impossible, thus $\Tin_n \subset \Tcut$. 

Finally, since $\oH_n^c \neq \emptyset$, consider $t_0 \in \arg\min_{t\notin \oH_n} \{ \ubx\nt - y\nt \}$.
By \eqref{eq:KKTlambdafactsOnCut}, we obtain:
\begin{equation}
\label{eq:ineqy0t0} y_{n,t_0} + \lambda_n < \ubx_{n,t_0} \  \Longleftrightarrow \ E_n  - \sum_{t\in\uH_n} \lbx\nt - \sum_{t\in \Tin_n} y\nt -\sum_{t\in\oH_n} \ubx\nt < |\Tin_n | (\ubx_{n,t_0} - y_{n,t_0})  \ 
\end{equation}
and thus:
\begin{align}
  \nonumber E_n \hh-\hh \sum_{t \in \Tcut^c} \lbx\nt \hh- \hh\sum_{t\in \Tcut} \ubx\nt &=   E_n -\hh \sum_{t\in\uH_n} \lbx\nt + \hh\hh\sum_{t\in\uH_n \cap \Tcut}\hh\hh \lbx\nt - \hh\hh\sum_{t\in \oH_n \cup \Tin_n}\hh\hh \ubx\nt  - \hh\hh\sum_{t\in\uH_n \cap \Tcut} \hh\hh\ubx\nt \\
  \nonumber &\qquad\qquad\qquad\qquad\qquad\qquad\qquad\qquad\qquad\qquad\qquad \quad \text{(as } \oH_n \cup \Tin_n \subset \Tcut \text{)} \\
& \label{eq:H0emptycontradiction} \!\!\!\!< \sum_{t\in \Tin_n} ( \ubx_{n,t_0} -y_{n,t_0})-  (\ubx\nt-y\nt) + \sum_{\uH_n \cap \Tcut} (\lbx\nt - \ubx\nt)  \quad \text{ (from \eqref{eq:ineqy0t0})}\\
& \label{eq:H0emptycontradiction2} \!\!\!\!\leq 0 \quad \text{ (from the definition of }t_0 \text{ and } \lbx\nt \leq  \ubx\nt\text{)} ,
\end{align}
which contradicts $n \notin \Ncut$. Thus we have shown that $ \forall n \notin \Ncut, \forall t \in \Tcut, \ \x\nt = \ubx\nt$. From \eqref{eq:condoptxE} and as $\forall t \in \Tcut, \nu_t>0$ from \eqref{eq:posit_nu_t}, we obtain $ \forall n \notin \Ncut, \forall t \in \Tcut, \ \y\nt=x\nt+\nu_t > \ubx\nt$. This terminates the proof for \Cref{claimFactcut1}.
\end{proof}
\begin{proof}[Proof of \Cref{claimFactcut2}]
Let us consider $\Tcut' \eqDef \{t | \nu_t>0 \}$. We already showed that $\T \subset \Tcut'$ in \eqref{eq:posit_nu_t}.
 To prove the other inclusion and obtain \Cref{claimFactcut2}, we observe that, considering $n \notin\Ncut$ (nonempty as shown independently in \Cref{claimFactcutNonemptysets}) and considering  $\Tcut'$ instead of $\Tcut$, all the facts established in the proof of \Cref{claimFactcut1} above also hold: by contradiction of $\forall t\in \Tcut'$, $x\nt = \ubx\nt$, we first obtain $\lambda_n  < 0$ as in \eqref{eq:lambda_neg_from_contrad}. 
 Then, as for $t" \in \Tin_n$, we have $\nu_{t"}=-\lambda_n >0$, we get $\Tin_n \cap \Tcut'^c= \emptyset$. Finally the same sequence of inequalities  as (\ref{eq:H0emptycontradiction}-\ref{eq:H0emptycontradiction2}) show  a contradiction. 
Consequently, for each $t\in \Tcut'$, $x\nt = \ubx\nt$ and $y\nt = x\nt+\nu_t > \ubx\nt $, thus $t\in \Tcut$ and $\Tcut'\subset \Tcut$. 
\end{proof}
\begin{proof}[Proof of \Cref{claimFactcut4}]
Suppose on the contrary that there exists $n\in\Ncut$ such that $\lambda_n\geq 0$. For $t\in \Tin_n$, we have $\nu_t=-\lambda_n \leq 0$, thus, $\Tin_n \subset \Tcut^c$.  
Then, if $t\in \Tcut$ and if $x\nt < \ubx\nt$, we would have: 
\begin{equation*}
x\nt = \max( \lbx\nt, y\nt+ \lambda_n) \geq x\nt +0+ \nu_t > x\nt \ ,\end{equation*}
which is impossible, thus $x\nt =\ubx\nt$, and $\Tcut \subset \oH_n$. As we show independently in \Cref{claimFactcutNonemptysets} that $\Tcut \neq \emptyset$, we know $\oH_n \neq \emptyset$. Let us consider $t_0 \in \arg\min_{t\notin \uH_n} \{ y\nt-\lbx\nt\}$. 
By \eqref{eq:KKTlambdafactsOnCut}, we obtain:
\begin{equation}
\label{eq:ineqy0t02} y_{n,t_0} + \lambda_n > \lbx_{n,t_0} \  \Longleftrightarrow \ E_n  - \sum_{t\in\uH_n} \lbx\nt - \sum_{t\in \Tin_n} y\nt -\sum_{t\in\oH_n} \ubx\nt > |\Tin_n | (\lbx_{n,t_0} - y_{n,t_0})  \ 
\vspace{-0.5cm}
\end{equation}
and thus:
\begin{align}
  \nonumber E_n - \hh \sum_{t\in \Tcut^c} \lbx\nt - & \hh \sum_{\Tcut} \ubx\nt =  E_n- \sum_{t\in\uH_n} \lbx\nt -\hh\hh\hh\hh \sum_{t\in \Tcut^c\cap \oH_n} \hh\hh\hh\lbx\nt  -\hh \sum_{t\in\Tin_n} \hh\lbx\nt  -\hh \sum_{t\in\oH_n}\hh \ubx\nt  + \hh\hh\hh\hh \sum_{t\in \Tcut^c \cap \oH_n} \hh\hh\hh\hh \ubx\nt \\
\nonumber  & \qquad \qquad \qquad \qquad\qquad\qquad \text{(as } \Tcut \subset \oH_n \text{)} \\
\label{eq:contradT0N0-3}& \!\!\!\!\!\!\!\!\!\!\!\!\!\!\!\!\!\!>\sum_{t\in\Tin_n}\left((y\nt-\lbx\nt) - (\y_{n,t_0}-\lbx_{n,t_0}) \right) + \sum_{t\in \Tcut^c \cap \oH_n} \ubx\nt-\lbx\nt \quad \text{(from \eqref{eq:ineqy0t02}) }\\
\label{eq:contradT0N0-4}& \!\!\!\!\!\!\!\!\!\!\!\!\!\!\!\!\!\!\geq 0  \ \ \quad \text{ (from the definition of }t_0 \text{ and } \lbx\nt \leq  \ubx\nt\text{)} ,
\end{align}
which contradicts $n\in\Ncut$ and terminates the proof of \Cref{claimFactcut4}. %
\end{proof}
\begin{proof}[Proof of \Cref{claimFactcut3}]
From \ref{claimFactcut2}, we know that $\Tcut^c=\{t | \nu_t \leq 0 \}$, thus, if $t \notin \Tcut$ and $n \in \Ncut$, if $x\nt > \lbx\nt$ then we would have $x\nt \leq y\nt+\lambda_n = x\nt + \nu_t + \lambda_n < x\nt$, which is a contradiction. %
\end{proof}
\begin{proof}[Proof of \Cref{claimFactcutNonemptysets}]
From $\sum_t \nu_t = 0$, we see that if $\Tcut= \emptyset$, then this means that $\nu_t=0$ for all $t\in\Tall$, and thus  $\xP=\xE$ which is a contradiction. Thus there exists $t_0$  such that $\nu_{t_0}>0$ and because $\sum_{t\in\Tcut} \nu_t=0$, there exists $t'_0$ such that  $\nu_{t'_0}<0$.

If $\Ncut= \emptyset$, then using \ref{claimFactcut1}, we would have for all $n$, $y_{n,t_0} > \ubx_{n,t_0}$ and thus $p_{t_0} > \sum_{n\in\Nall} \ubx_{n,t_0}$, which contradicts the aggregate upper bound constraint $\forall t, \ p_{t} \leq \sum_{n\in\Nall} \ubx_{n,t} $.

If $\Ncut^c = \emptyset$, then using \ref{claimFactcut3}, we would have for all $n$, $y_{n,t_0'} < \lbx_{n,t_0'} $ and thus $p_{t_0'} <\sum_{n\in\Nall}\lbx_{n,t_0'}$, which contradicts the aggregate lower bound constraint $\forall t, \ p_{t} \geq \sum_{n\in\Nall} \lbx_{n,t} $.
\end{proof}

\section{Proof of \texorpdfstring{\Cref{lemm:factsCutFinite}}{}}
\label{app:proofCutFinite}
\begin{proof}[Proof of \Cref{claimFactcutFinite1}]
From $\xx\Kexp= P_{\X}(\yy\Kexpp)$ and $\yy\Kexp=P_{\Y}(\xx\Kexp)$, we obtain, similarly to \eqref{eq:condoptxE}:
\begin{equation} \label{eq:condoptxFinite}
\forall n\in \Nall,  \forall t \in \Tall,\   0 = (x\nt\Kexp - y\nt\Kexpp) - \lambda_n\Kexp \hh - \umu\nt\Kexp  \hh + \omu\nt\Kexp   \ \text{and } \ y\nt\Kexp = x\nt\Kexp + \nu_t\Kexp .
\end{equation}
where the Lagrangian multipliers $\lambda_n\Kexp, \umu\nt\Kexp , \omu\nt\Kexp$ (resp. $\nu_t\Kexp$)  are associated to the quadratic problem characterizing the projections $ P_{\X}(\yy\Kexpp)$  (resp. $\yy\Kexp=P_{\Y}(\xx\Kexp)$). 
We obtain equalities similar to (\ref{eq:KKT1factsOnCut}, \ref{eq:KKTlambdafactsOnCut}).
We proceed as for \Cref{prop:factsOncut}\ref{claimFactcut1} and suppose that there exists $n \notin \Ncut$ 
 and $\hatt \in \Tcut$ such that $\x\nth < \ubx\nth$. Then, as  $\norm{\yy\Kexp-\yyinf } \leq \tfrac{\epscvg}{1-\rho}$ and $\sum_{t\in\Tall} \yy_t\Kexp = \sum_{t\in\Tall} \yyinf_t$, we have for each $n\in\Nall, t\in\Tall$, $|y\nt\Kexp-y\nt^\infty| \leq \tfrac{\epscvg}{2(1-\rho)} $, and thus we get:
 \begin{equation}
 \label{eq:findBoundLambdaK}
 \begin{split}
& \ubx\nth \geq \x\nth\Kexp \geq y\nth\Kexpp + \lambda_n\Kexp \geq \y\nth^\infty - \tfrac{\epscvg}{2(1-\rho)} + \lambda_n\Kexp = x\nth^\infty+ \nu_{\hatt}^\infty - \tfrac{\epscvg}{2(1-\rho)} + \lambda_n\Kexp \\
&  \Longrightarrow \lambda_n\Kexp  <  \tfrac{\epscvg}{2(1-\rho)} v -  \nu_{\hatt}^\infty  < \tfrac{\Bseuil\epscvg}{2}  - 2 \Bseuil\epscvg = - \tfrac{3}{2}\Bseuil\epscvg  \ ,  
\end{split}
 \end{equation}
as $\nu_{\hatt}^\infty \geq \lnu  > 2 \Bseuil\epscvg$. Let us now consider $t'\in {\Tin_n}\Kexp \cup \oH_n\Kexp$, then:
\begin{equation} \label{eq:boundNuK}
\nu_{t'}\Kexp = y_{n,t'}\Kexp-x_{n,t'}\Kexp \geq y_{n,t'}\Kexp-y_{n,t'}\Kexp  - \lambda_n\Kexp > -\tfrac{\epscvg}{2} + \tfrac{3}{2}\Bseuil\epscvg  > \Bseuil\epscvg + \tfrac{\epscvg}{2}(\Bseuil-1) \geq \Bseuil\epscvg   \ ,
\end{equation}
which shows that $ t' \in \Tcut\Kexp=\Tcut^\infty$ and thus ${\Tin_n}\Kexp \cup \oH_n\Kexp \subset \Tcut$. 
Then, the same sequence of inequalities  as (\ref{eq:ineqy0t0}, \ref{eq:H0emptycontradiction}, \ref{eq:H0emptycontradiction2}) applied to $\yy\Kexpp$ gives a contradiction to $n\notin\Ncut$.
\end{proof}

\begin{proof}[Proof of \Cref{claimFactcutFinite2}]
The proof of \Cref{claimFactcutFinite2} is symmetric to the one of \Cref{claimFactcutFinite1}: if we suppose that there exists $n\in\Ncut$ and $\hat{t}\notin \Tcut$ such that $\x\nth\Kexp> \lbx\nth$, we obtain, symmetrically to \eqref{eq:findBoundLambdaK}, that $\lambda_n\Kexp \geq - \tfrac{\epscvg}{2(1-\rho)}$. Then, considering $t' \in \uH_n\Kexp \cup {\Tin_n}\Kexp$, we show, symmetrically to \eqref{eq:boundNuK}, that $\nu_{t'}\Kexp < \Bseuil\epscvg$ i.e.\  $t' \notin \Tcut$ and thus $\uH_n\Kexp \cup {\Tin_n}\Kexp \subset \Tcut^c$. 
We conclude by obtaining a contradiction to $n\in\Ncut$ by the same sequence of inequalities as (\ref{eq:ineqy0t02}, \ref{eq:contradT0N0-3}, \ref{eq:contradT0N0-4}).
\end{proof}

\end{appendix}


\begin{small}
\bibliographystyle{siam}
\bibliography{../../../USEFULPAPERS/Biblio_complete/shortJournalNames.bib,../../../USEFULPAPERS/Biblio_complete/biblio1,../../../USEFULPAPERS/Biblio_complete/biblio2,../../../USEFULPAPERS/Biblio_complete/biblio3,../../../USEFULPAPERS/Biblio_complete/biblio4,../../../USEFULPAPERS/Biblio_complete/biblio5,../../../USEFULPAPERS/Biblio_complete/biblioBooks}
\end{small}
\end{document}